\documentclass{amsart}

\usepackage{amsmath,amsthm,amscd,amsfonts,wasysym,amssymb,epic,eepic,bbm,tikz-cd,mathtools,multicol,enumitem,mathrsfs,comment}
\usepackage[pagebackref,colorlinks=true,linkcolor=blue,citecolor=blue]{hyperref}
\usepackage[noabbrev]{cleveref}
\usepackage{MnSymbol}
\makeatletter
\newcommand\myitem[1][]{\item[#1]\refstepcounter{dummy}\def\@currentlabel{#1}}
\makeatother

\allowdisplaybreaks
\newtheorem{thm}{Theorem}[section]
\newtheorem*{thm*}{Theorem}

\newtheorem{lem}[thm]{Lemma}

\newtheorem{prop}[thm]{Proposition}

\newcounter{remarkscounter}

\numberwithin{equation}{section}

\newcommand{\SL}{\mathrm{SL}}

\newcommand{\Spec}{\mathrm{Spec}}

\newcommand{\Lc}{\mathcal{L}}

\newcommand{\C}{\mathbf{C}}

\newcommand{\Hom}{\mathrm{Hom}}

\makeatletter
\newcommand{\rprod}{%
  \DOTSB             
  \mathop{\mathpalette\rprod@\relax}\slimits@}
\newcommand{\rprod@}[2]{%
  \ooalign{$\m@th#1\prod$\cr$\m@th#1\coprod$\cr}}
\makeatother

\newcommand{\quash}[1]{}

\theoremstyle{definition}
\newtheorem{defn}[thm]{Definition}
\newtheorem{lemma}[thm]{Lemma}

\renewcommand{\bar}{\overline}
\numberwithin{equation}{subsection}

\renewcommand{\hat}{\widehat}

\def\C{{\mathbf{C}}}

\def\Hom{{\mathrm {Hom}}}

\def\End{{\mathrm {End}}}

\def\ker{{\mathrm {ker}\,}}
\def\Ker{{\mathrm {Ker}\,}}

\def\Spec{{\mathrm {Spec}}}

\def\dim{{\mathrm{dim}}}

\def\SL{{\mathrm{SL}}}

\def\C{{\mathbf{C}}}

\allowdisplaybreaks

\begin{document}
\linespread{1.2}

\title{Cohomology on Cotangent Bundles of Partial Flag Varieties in Type A}

\author{Nikolay Grantcharov}
\address{Department of Mathematics\\
University of Georgia\\
Athens, GA 30602}
\email{nikolayg@uga.edu}
\begin{abstract}
Let $G=\mathrm{SL}_n(\mathbf C)$ and let $P\subset G$ be a standard parabolic subgroup with Levi factor $L$. For a $G$-dominant weight $\lambda$, consider the vector bundle on $T^*(G/P)$ obtained by pulling back the vector bundle on $G/P$ associated to the irreducible $L$-module $V_L(\lambda)^*$. We express its cohomology as a direct limit of the cohomology of certain line bundles on a Bott--Samelson variety associated to an affine Kac-Moody group. This comparison yields vanishing of higher cohomology and shows that the global sections are generated over $\mathbf C[\mathfrak g^*]$ by their degree-zero part $V_G(\lambda)^*$. Using these results together with the Braverman--Kazhdan intertwiners constructed in earlier joint work with A. Slipper, we give an explicit generating set for $\mathbf C[T^*(\mathrm{SL}_n/[P,P])]$. Finally, in an appendix joint with Tom Gannon, we combine these results to show the affinization $\Spec(\C[T^*(SL_n/[P,P])])$ has terminal singularities.
\end{abstract}
\maketitle

\setcounter{tocdepth}{1}
\tableofcontents

\section{Introduction}
Let $\mathbf{C}$ be the ground field. Fix $G=SL_n(\mathbf{C})$ and let $P$ be a standard parabolic subgroup of $G$ with Levi subgroup $L$. A character $\lambda\in\mathbf{X}_{G,P}:=\Hom(L,\mathbf{G}_m)$ of $L$ extends trivially to a character of $P$, giving rise to a homogeneous line bundle $\Lc_{G/P}(\lambda):=G\times^P\mathbf{C}_\lambda$ over the partial flag variety $G/P$. More generally, for a finite-dimensional $L$-module $M$, we may associate a homogeneous vector bundle $\mathcal{L}_{G/P}(M):=G\times^PM$ over $G/P$. Since $L$ is reductive, an irreducible $L$-module $M:=V_L(\lambda)$ is indexed by an $L$-dominant highest weight $\lambda$. The Borel-Weil-Bott theorem then gives a complete description of all cohomology groups of $\mathcal{L}_{G/P}(V_L(\lambda)^*)$, where $V_L(\lambda)^*$ is the dual $L$-representation. Let $\mathbf{X}^+$ denote the set of $G$-dominant characters. Thus when $\lambda\in\mathbf{X}^+$, all higher cohomology groups of $\mathcal{L}_{G/P}(V_L(\lambda)^*)$ vanish and the zeroth cohomology group is $V_G(\lambda)^*$.

In this article, we first study the cohomology groups of the pullback $\mathcal{L}_{T^*(G/P)}(V_L(\lambda)^*)$ of the vector bundles $\mathcal{L}_{G/P}(V_L(\lambda)^*)$ along the projection $\pi:T^*(G/P)\rightarrow G/P$ for any $\lambda\in \mathbf{X}^+.$ We prove two main results:

\begin{thm*}[Theorem \ref{vanishing of higher cohomology}]
Let $G=SL_n(\mathbf{C})$ and $P$ a standard parabolic. For every $\lambda\in\mathbf X^+$ and every $i\geq1$,
\[
H^i\bigl(
T^*(G/P),
\mathcal L_{T^*(G/P)}(V_L(\lambda)^*)
\bigr)=0.
\]
\end{thm*}

\begin{thm*}[Theorem \ref{generation of zeroeth cohomology}]
Let $G=SL_n(\mathbf{C})$ and $P$ a standard parabolic. For every $\lambda\in\mathbf X^+$, the natural map
\begin{equation*}
\phi_\lambda:
\C[\mathfrak g^*]\otimes
H^0\bigl(G/P,\mathcal L_{G/P}(V_L(\lambda)^*)\bigr)
\longrightarrow
H^0\bigl(
T^*(G/P),
\mathcal L_{T^*(G/P)}(V_L(\lambda)^*)
\bigr),
\end{equation*}
is a surjective homomorphism of graded $G$-equivariant
$\C[\mathfrak g^*]$-modules.
\end{thm*}

Computing the cohomology groups of the vector bundles, $\mathcal{L}_{T^*(G/P)}(V_L(\lambda)^*)$ with dominant weights $\lambda\in\mathbf{X}^+$, over $T^*(G/P)$ has a long history. Graham first proved both theorems for $P=B$ and minuscule weights \cite{Graham1992} by building on work of Hesselink \cite{Hesselink1980}. Broer then proved them for $P=B$ and arbitrary dominant weights using spectral sequences \cite{Broer1993} and later established the higher-cohomology vanishing for arbitrary parabolics and line bundles over $T^*(G/P)$, corresponding to $\lambda\in\mathbf{X}_{G,P}^+:=\mathbf{X}^+\cap\mathbf{X}_{G,P}$, using Grauert--Riemenschneider in \cite{Broer1994}. For arbitrary parabolics in type $A$, Hague proved the higher-cohomology vanishing for regular dominant weights \cite[Theorem 4.23]{HagueBK}. Further related results and alternative approaches, including methods via Frobenius splitting in positive characteristic, appear in \cite{AndersenJantzen1984, Broer1997,Hague2013}. Finally, Kato's
vanishing theorem for type-$A$ root ideals implies the higher-cohomology
vanishing for every $\lambda\in\mathbf{X}^+$ and every
parabolic $P$ \footnote{The higher cohomology vanishing for $T^*(G/P)$ is only stated for $\lambda\in\mathbf{X}_{G,P}^+$ in \cite[Cor 5.4]{Kato2025}} by
\cite[Theorem 5.1]{Kato2025} and Lemma \ref{identify with T(G/P)}. To the author's knowledge, however, neither the higher-cohomology vanishing nor the generation of zeroth cohomology had previously been stated in the present form for arbitrary $G$-dominant weights $\lambda\in\mathbf{X}^+$ and arbitrary parabolic subgroups $P$.

These vector bundles also have a categorical interpretation. When $P=B$, their derived pushforwards along the Springer resolution are the Andersen--Jantzen sheaves \cite{AndersenJantzen1984}. Bezrukavnikov showed that the dominant Andersen--Jantzen sheaves form a quasi-exceptional set generating $D^b\operatorname{Coh}^G(\mathcal N)$ \cite{Bezrukavnikov2003}. For general $P$, Achar and Riche showed that the larger family of homogeneous bundles indexed by all $L$-dominant weights generates $D^b\operatorname{Coh}^{G\times\mathbf G_m}(T^*(G/P))$; their mutations give the exceptional objects defining the parabolic exotic $t$-structure studied in \cite{AcharRiche2018,AcharCooneyRiche2018}. Since the Richardson orbit closure $\overline{\mathcal O}_P$ is affine, Theorem~\ref{vanishing of higher cohomology} implies that the derived pushforwards of the bundles along the moment map $T^*(G/P)\to\overline{\mathcal O}_P$ are concentrated in degree zero. These coherent sheaves may therefore be viewed as parabolic analogues of Andersen--Jantzen sheaves, and Theorem~\ref{generation of zeroeth cohomology} gives explicit generators for their global sections as $\mathbf C[\mathfrak g^*]$-modules.

Our proofs of both theorems are based on S. Kato's wonderful idea to compactify the closely related variety $G\times^B\mathfrak{u}_P$, which is the pullback of $T^*(G/P)$ along $G/B\to G/P$. Although Theorem~\ref{vanishing of higher cohomology} follows from \cite[Theorem 5.1]{Kato2025}, and \cite[Corollary 5.13]{Kato2025} provides an important step toward Theorem~\ref{generation of zeroeth cohomology}, we give complete, self-contained proofs of all results because Kato's construction and arguments simplify substantially in our setting.\footnote{Kato works more generally with root ideals replacing $\mathfrak{u}_P$.}  In particular, we realize Kato's compactification in terms of flags of lattices \cite{Kamnitzer2024} and then show directly that it is a tower of projective vector bundles whose ranks are determined by the block sizes of $P$. This description has the advantage of making the resolution by generalized Bott--Samelson varieties transparent. We then construct line bundles on the generalized Bott--Samelson variety whose cohomology agrees with that of the vector bundles on $T^*(G/P)$ considered above; see Equation \ref{comparison of cohomology}. The main novelty of our argument is that both main theorems then follow naturally from the Borel--Weil theorem for generalized Bott--Samelson varieties, Theorem \ref{Bott-Borel for Bott-Samelson} and \cite{LakshmibaiLittelmannMagyar2002}. We conjecture that the analogue of Theorem~\ref{generation of zeroeth cohomology} holds for arbitrary reductive group $G$ whenever the Springer map $T^*(G/P)\to\overline{\mathcal O}_P$ is birational and the associated Richardson orbit closure $\overline{\mathcal O}_P$ is normal. These hypotheses are necessary since already for the zero weight, the assertion fails if the Springer map has degree greater than one or if $\overline{\mathcal O}_P$ is not normal.

These cohomological results provide the first of two key ingredients in establishing two new properties for the variety $\overline{T^*(SL_n/[P,P])}^{\mathrm{aff}}$. The other ingredient is the family of Braverman--Kazhdan intertwiners constructed in \cite{GrantcharovSlipper2026}. These intertwiners produce isomorphisms, for any pair of parabolics $P,P^w$ whose Levi subgroups are conjugate by $w\in S_k$, where $k$ is the number of blocks of $P$, of
\begin{equation*}
    \Phi_w:\overline{T^*(SL_n/[P,P])}^{\text{aff}}\rightarrow \overline{T^*(SL_n/[P^w,P^w])}^{\text{aff}}.
\end{equation*}
In \cite{GrantcharovSlipper2026}, we also showed that these intertwiners satisfy the Coxeter relations for $S_k$ and a twisted $SL_n\times L^{\mathrm{ab}}$-equivariance relation. Write $\Phi_w^*$ for the induced map on coordinate algebras. Using the equivariance properties of the intertwiners together with Theorem~\ref{generation of zeroeth cohomology}, we obtain the following explicit generating set for the algebra $\mathbf C[T^*(SL_n/[P,P])]$, which is a parabolic analogue of \cite[Lemma 3.6.2]{GinzburgRiche2015}: 

\begin{thm*}[Theorem \ref{generators of R}]
The algebra $\mathbf{C}[T^*(SL_n/[P,P])]$ is generated by the image of the comoment map $\mu_{\mathfrak{sl}_n\times(\mathfrak{l}^{\mathrm{ab}})^*}: \mathbf{C}[\mathfrak{sl}_n^*\times(\mathfrak{l}^{\text{ab}})^*]\rightarrow\mathbf{C}[T^*(SL_n/[P,P])]$ and by $\Phi_w^*\bigl(\mathbf C[SL_n/[P^w,P^w]]\bigr)$, for each $w\in S_k$.
\end{thm*}
When $P = B$, the explicit generating set of Theorem \ref{generators of R} played an important role in Gannon's proof of the \textit{Ginzburg-Kazhdan conjecture} that $\overline{T^*(G/U)}^{\text{aff}}$ has symplectic singularities for any reductive group $G$, where $U=[B,B]$ is the unipotent radical of a Borel subgroup $B\subset G$. His proof generalizes an insight of Jia \cite{Jia2025}, who first proved the Ginzburg-Kazhdan conjecture in the case $G = SL_n$ by proving that the codimension of the singular locus of $\overline{T^*(SL_n/U)}^{\text{aff}}$ is at least four. From this fact, along with its analogue for general reductive groups proved in \cite{Gannon2024}, it follows essentially immediately that $\overline{T^*(G/U)}^{\text{aff}}$ moreover has \textit{terminal} symplectic singularities. 

Recently, it has been shown \cite{FuLiu2025} that $\overline{T^*(G/[P, P])}^{\text{aff}}$ has symplectic singularities for an arbitrary reductive group $G$. Using this, in an appendix written jointly with Tom Gannon, we prove that, when $G = \SL_n$, $\overline{T^*(G/[P, P])}^{\text{aff}}$ moreover has \textit{terminal} symplectic singularities: 

\begin{thm*}[Theorem \ref{terminal singularities}]
    The variety $\overline{T^*(SL_n/[P,P])}^{\text{aff}}$ has terminal singularities.
\end{thm*}

\noindent\textbf{Acknowledgments.} The author thanks Shrawan Kumar, Scott Larson, Alberto San Miguel Malaney, Dan Nakano, and Aaron Slipper for many stimulating and productive discussions. 
The author is also especially grateful to Bill Graham for going through the arguments of the author's main results and providing lots of thoughtful feedback throughout.

AI Disclosure: In preparing this article, the author used ChatGPT 5.6 for proofreading. It was also used as an exploratory aid in comparing Kato's compactification of $G\times^B\mathfrak{u}_P$ \cite{Kato2025} with the generalized Bott-Samelson varieties. All text is human written and verified, and the author assumes full responsibility for its content.

\section{Cohomology on cotangent bundles of partial flag varieties}
In this section, we establish the cohomological results for homogeneous
vector bundles on cotangent bundles of partial flag varieties that will
be needed for the subsequent results. Our main results are the higher
cohomology vanishing in Theorem
\ref{vanishing of higher cohomology} and the generation theorem for
global sections in Theorem
\ref{generation of zeroeth cohomology}. As mentioned in the introduction,
our strategy is to first compactify the variety
$G\times^B\mathfrak{u}_P$, which is the pullback of
$T^*(G/P)\simeq G\times^P\mathfrak{u}_P$ along the projection
$G/B\rightarrow G/P$. By Borel--Weil and a Leray spectral sequence argument,
(Lemma \ref{identify with T(G/P)}), the relevant cohomologies on these
two varieties are identified. The compactification is denoted by $Y_P$,
and it has an explicit description in terms of lattice flags
(Definition \ref{Y_P definition}). Next, we show that a Bott--Samelson
variety $Z_P$ (\ref{Bott-Samelson for P}) for an affine Kac--Moody group
admits a proper birational morphism to $Y_P$. Finally, we show that the
vector bundles
$\mathcal{L}_{T^*(G/P)}(V_L(\lambda)^*)$, for
$\lambda\in\mathbf{X}^+$, correspond to line bundles on $Z_P$
associated with nonnegative markings (Lemmas
\ref{restriction of M_P to open} and \ref{pullback along rho}). This,
together with localization and comparison with the case $P=B$, allows
us to conclude both main theorems using the Borel--Weil theorem and
higher-cohomology vanishing for Bott--Samelson varieties
(Theorem \ref{Bott-Borel for Bott-Samelson}).

Fix $G=\mathrm{SL}_n(\C)$ and a maximal torus and a Borel subgroup
$T\subset B\subset G$. Let $\Phi^+$ be the corresponding set of positive
roots, and let $\Delta\subset\Phi^+$ be the set of simple roots. Let
$P=P_I\supset B$ be the standard parabolic subgroup associated with a
subset $I\subset\Delta$, let $L\supset T$ be its standard Levi subgroup,
and let $U_P$ be its unipotent radical. We write $\mathfrak g$, $\mathfrak p$,
and $\mathfrak u_P$ for their Lie algebras. Let $W$ be the Weyl group and $w_0\in W$ be the longest element.

Denote the character lattice and the set of dominant weights, respectively, by
\begin{align*}
    \mathbf X&:=\operatorname{Hom}_{\mathrm{alg}}(T,\mathbf G_m)\\
    \mathbf X^+&:= \{\lambda\in\mathbf X:
\langle\lambda,\alpha^\vee\rangle\geq0
\text{ for every }\alpha\in\Delta\}
\end{align*}
Similarly, denote the $P$-character lattice and set of $G$-dominant weights which are $P$-characters by
\begin{align*}
\mathbf X_{G,P}
&:=
\operatorname{Hom}_{\mathrm{alg}}(P,\mathbf G_m)
\simeq
\operatorname{Hom}_{\mathrm{alg}}(L/[L,L],\mathbf G_m).\\
\mathbf X_{G,P}^+&:=\mathbf X_{G,P}\cap\mathbf X^+.
\end{align*}
Note that by restricting to $T$, we may identify
$\mathbf X_{G,P}
=
\{\lambda\in\mathbf X:
\langle\lambda,\alpha^\vee\rangle=0
\text{ for every }\alpha\in I\}.$
Thus, if the parabolic $P$ corresponds to the partition $0=d_0<d_1<\dots<d_k=n$, i.e., the block sizes of its Levi are $s_i:=d_i-d_{i-1},1\leq i\leq k$, then we may identify 
$$\mathbf{X}_{G,P}=\bigoplus_{i=1}^{k-1}\mathbf{Z}\varpi_{d_i},\quad \mathbf{X}_{G,P}^+=\bigoplus_{i=1}^{k-1}\mathbf{Z}_{\geq0}\varpi_{d_i}$$
where $\varpi_i,1\leq i<n$ are the fundamental weights of $G$.

For a finite-dimensional $P$-module $M$, define
\[
\mathcal{L}_{G/P}(M):=G\times^P M,\quad \mathcal{L}_{G/P}(M)^*:=G\times^PM^*.
\]
We will reserve the notation $\mathcal{L}_X(\lambda)$ for line bundles
on a variety $X$ and $\mathcal{L}_X(M)$ for vector bundles on a variety
$X$.

Next, using the Killing form $\kappa:\mathfrak{sl}_n\xrightarrow{\sim}\mathfrak{sl}_n^*$, we may canonically identify $\mathfrak{u}_P\simeq(\mathfrak{g}/\mathfrak{p})^*$ in a $P$-equivariant way. Thus we have $SL_n$-equivariant morphisms
$$T^*(G/P)\simeq G\times^P \mathfrak{p}^\perp\simeq G\times^P\mathfrak{u}_P.$$

Let $$\pi_P:T^*(G/P)\rightarrow G/P$$ be the standard projection. Thus, for a finite-dimensional $P$-module $M$ we may consider the corresponding vector bundles on $T^*(G/P)$ by
\begin{equation}\label{vector bundle on T^*(G/P)}
    \mathcal{L}_{T^*(G/P)}(M):=\pi_P^*(\mathcal{L}_{G/P}(M))\simeq G\times^P(\mathfrak{u}_P\times M).
\end{equation}

\begin{lemma}\label{identify with T(G/P)}
Let $\lambda\in\mathbf X^+$ and set $A=V_L(\lambda)$, regarded as a
$P$-module via $P\twoheadrightarrow L$. Let 
\[
\pi_B:G\times^B\mathfrak u_P\rightarrow G/B
\qquad\text{and}\qquad
\pi_P:T^*(G/P)\rightarrow G/P
\]
be the natural projections. Then, for every $i\geq0$, there are
canonical $G$-module isomorphisms
\begin{align*}
H^i\bigl(
G\times^B\mathfrak u_P,
\pi_B^*\mathcal L_{G/B}(\lambda)^*
\bigr)
&\simeq
H^i\bigl(
T^*(G/P),
\pi_P^*\mathcal L_{G/P}(A^*)
\bigr),\\
H^i\bigl(
G/B,
\mathcal L_{G/B}(\lambda)^*
\bigr)
&\simeq
H^i\bigl(
G/P,
\mathcal L_{G/P}(A^*)
\bigr).
\end{align*}
\end{lemma}

\begin{proof}
Set $M_j=S^j(\mathfrak u_P^*)$. Since $\pi_B$ and $\pi_P$ are affine morphisms, the affine projection formula gives canonical isomorphisms of graded $G$-modules
\begin{align*}
H^i\bigl(
G\times^B\mathfrak u_P,
\pi_B^*\mathcal L_{G/B}(\lambda)^*
\bigr)
&\simeq
\bigoplus_{j=0}^{\infty}H^i\bigl(
G/B,
\mathcal L_{G/B}(M_j\otimes\mathbf C_{-\lambda})
\bigr),\\
H^i(T^*(G/P),\pi_P^*\mathcal L_{G/P}(A)^*)&\simeq\bigoplus_{j=0}^\infty H^i(G/P, \mathcal{L}_{G/P}(M_j\otimes A^*))
\end{align*}
Since $M_j$ is a $P$-module and $\lambda\in\mathbf{X}^+$, the tensor identity \cite[I.4.8]{Jan03} and Borel--Weil on
$P/B\simeq L/(L\cap B)$ give
\[
H^q\bigl(
P/B,
\mathcal L_{P/B}(M_j\otimes\mathbf C_{-\lambda})
\bigr)
\simeq
\begin{cases}
M_j\otimes A^*,&\text{if }q=0,\\
0,&\text{if }q>0.
\end{cases}
\]
Therefore, the spectral sequence for transitivity of induction
\cite[I.4.5(c)]{Jan03}, equivalently the Leray spectral sequence
for $q_0:G/B\rightarrow G/P$, degenerates and gives
\[
H^i\bigl(
G/B,
\mathcal L_{G/B}(M_j\otimes\mathbf C_{-\lambda})
\bigr)
\simeq
H^i\bigl(
G/P,
\mathcal L_{G/P}(M_j\otimes A^*)
\bigr).
\]
This identifies the degree $j$ components for every $j\geq0$, proving
the first isomorphism. Taking $j=0$ proves the second.
\end{proof}
We now state the main theorems.

\begin{thm}\label{vanishing of higher cohomology}
Let $G=SL_n$ and let \(P\supset B\) be a parabolic subgroup and let
\(\lambda\in\mathbf{X}^+\). Let $V_L(\lambda)$ denote the corresponding irreducible $L$-module of highest weight $\lambda$. Then, for all \(i\geq 1\),
\[
H^i\bigl(
T^*(G/P),
\mathcal L_{T^*(G/P)}( V_L(\lambda)^*)
\bigr)
=0.
\]
\end{thm}

\begin{thm}\label{generation of zeroeth cohomology}
Let $G=SL_n$ and let \(P\supset B\) be a parabolic subgroup and let
\(\lambda\in\mathbf{X}^+\). Let $V_L(\lambda)$ denote the corresponding irreducible $L$-module of highest weight $\lambda$. Let $\mu:T^*(G/P)\rightarrow\mathfrak g^*$ be the moment map. Then multiplication defines a map
\begin{equation}\label{surjection}
\begin{aligned}
\phi_\lambda:
\mathbf C[\mathfrak g^*]\otimes
H^0(G/P,\mathcal L_{G/P}(V_L(\lambda)^*))
&\longrightarrow
H^0(T^*(G/P),\mathcal L_{T^*(G/P)}(V_L(\lambda)^*)),\\
f\otimes s&\longmapsto\mu^*(f)\pi_P^*(s).
\end{aligned}
\end{equation}
This map is a surjective homomorphism of graded $G$-equivariant
$\mathbf C[\mathfrak g^*]$-modules.
\end{thm}

\subsection{Lattice compactification}
Let $P$ be the standard parabolic of $SL_n$ corresponding to the partition $0=d_0<d_1<\dots<d_k=n$. So there are $k$ blocks of the Levi and their block sizes are $d_i-d_{i-1}, 1\leq i\leq k$. Let 
$$b_P(0):=0,\;\;\;b_P(j):=d_{a-1}\quad\text{if }d_{a-1}<j\leq d_a,$$
be the function which extracts the left endpoint of the block. Let $V=\mathbf{C}^n$ and denote the standard basis by $\{e_1,\dots,e_n\}$. Write the standard flag $F_j^0:=\text{span}\{e_1,\dots,e_j\}$.
We may identify the parabolic nilradical as:
\begin{equation*}
    \mathfrak{u}_{P}=\{X\in\End(\mathbf{C}^n): X(F_j^0)\subset F_{b_P(j)}^0\}
\end{equation*}
\noindent Let $\text{Fl}(V)\simeq G/B$ denote the space of full flags $F_\bullet:=(0=F_0\subset F_1\subset\cdots\subset F_n=V)$ on $V$. Then we may identify 
\begin{equation}\label{flag description of cotangent bundle}
    G\times^B\mathfrak{u}_{P}=\{(F_\bullet,X)\in\text{Fl}(V)\times\End(\mathbf{C}^n): X(F_j)\subset F_{b_P(j)}\}.
\end{equation}
Let us now construct a smooth projective compactification of this space.

\begin{defn}\label{Y_P definition}
   Let $E=V\otimes \mathbf{C}[[t]]$. Define the variety
    \begin{equation}
        Y_P:=\{E=\Lambda_0\subset\Lambda_1\subset\dots\subset\Lambda_n\subset t^{-n}E:\dim(\Lambda_i/\Lambda_{i-1})=1, t\Lambda_i\subset\Lambda_{b_P(i)}\;\text{ for all } 1\leq i\leq n\}
    \end{equation}
We denote points of $Y_P$ by $\Lambda_\bullet=(\Lambda_0\subset\dots\subset\Lambda_n)$ and we call each $\mathbf{C}[[t]]$-module, $\Lambda_i$, a lattice.
\end{defn}
\noindent This variety has appeared in similar forms in \cite{Kamnitzer2024,Lusztig1981}. The following lemma also appears in \cite[Prop 2.1]{Kamnitzer2024}.
\begin{lem}\label{compactification lemma}
The variety $Y_P$ is a tower of smooth projective bundles. More precisely, if
$Y_P^{(j)}$ denotes the variety of truncated lattice flags
\begin{equation}\label{truncated Y_P}
Y_P^{(j)}:=\{E=\Lambda_0\subset\cdots\subset\Lambda_j:\dim(\Lambda_i/\Lambda_{i-1})=1,\;t\Lambda_i\subset\Lambda_{b_P(i)}\text{ for all }1\leq i\leq j\},
\end{equation}
with $E=V\otimes \mathbf{C}[[t]]$, then the map forgetting the last component,
\[
\pi_j:Y_P^{(j)}\longrightarrow Y_P^{(j-1)}\;\text{is a $\mathbf P^{\,n-j+b_P(j)}$-bundle.}
\]
In particular, $Y_P$ is smooth and
projective. Lastly, there is a closed embedding $\iota:Y_P\hookrightarrow Y_B.$
\end{lem}

\begin{proof}
Fix $\Lambda_0,\ldots,\Lambda_{j-1}$. Since
$b_P(j-1)\leq b_P(j)$, we have $t\Lambda_{j-1}\subset \Lambda_{b_P(j-1)}
\subset \Lambda_{b_P(j)}.$
Thus a permissible $\Lambda_j$ is equivalent to a line $\Lambda_j/\Lambda_{j-1} \subset t^{-1}\Lambda_{b_P(j)}/\Lambda_{j-1}.$
The dimension of the vector space $t^{-1}\Lambda_{b_P(j)}/\Lambda_{j-1}$ is $(n+b_P(j))-(j-1)=n-j+b_P(j)+1.$ Therefore the forgetful map is a
$\mathbf P^{\,n-j+b_P(j)}$-bundle. Finally, since $b_P(i)\leq i-1=b_B(i)$, every defining condition for $Y_P$
implies the corresponding condition for $Y_B$. These are closed
incidence conditions, so $Y_P\hookrightarrow Y_B$ is a closed immersion.
\end{proof}

In particular, the compactification $Y_P$ introduced here coincides with the one constructed by S. Kato in the case the root ideal comes from a parabolic subalgebra. See \cite[Proposition 3.7]{Kato2025}.

Set
\[
Q:=t^{-n}E/E\simeq\bigoplus_{m=1}^n t^{-m}V.
\]
Multiplication by $t$ induces a nilpotent endomorphism of $Q$, still
denoted by $t$. For a lattice flag in $Y_P$, set
$S_i:=\Lambda_i/E$. This identifies
\[
Y_P\simeq
\{0=S_0\subset\cdots\subset S_n\subset Q:
\dim S_i=i,\ tS_i\subseteq S_{b_P(i)}\}.
\]

Let $\mathcal{S}_j$ be the universal vector bundle on $Y_P$ whose fiber over a flag $S_\bullet\in Y_P$ is the space $S_j$. We have a standard projection to the $t^{-1}$-component 
\[
\operatorname{pr}_{-1}:Q\longrightarrow V,\qquad
\operatorname{pr}_{-1}\left(\sum_{m=1}^n t^{-m}v_m\right)=v_1.
\]

If we restrict this to $S_n$, we obtain a map of universal vector bundles of rank $n$ over $Y_P$:
\begin{equation}
    \theta_P=\text{pr}_{-1}\vert_{\mathcal{S}_n}:\mathcal{S}_n\rightarrow V\otimes\mathcal{O}_{Y_P}
\end{equation}
We denote the fiber of this map over a particular point $S_\bullet$ of $Y_P$ also by $\theta_P:S_n\rightarrow V$. 
Next, define
\begin{equation}
\mathcal A_P:=\det(\mathcal{S}_n)^\vee\otimes\det(V),
\qquad
\sigma_{P}:=\det(\theta_P)\in H^0(Y_P,\mathcal A_P).
\end{equation}
\noindent Let $D(\sigma_P)$ denote the non-vanishing locus inside $Y_P$ of the section $\sigma_P$.

\begin{lem}\label{principal open D(sigma_P)} We have
    \[
D(\sigma_P)\simeq G\times^B\mathfrak u_P.
\]
Moreover, under the closed immersion $\iota:Y_P\hookrightarrow Y_B$, we have  $\iota^*\mathcal A_B\simeq\mathcal A_P, \iota^*\sigma_B=\sigma_P.$
\end{lem}
\begin{proof}
Let $S_\bullet\in D(\sigma_P).$ Then $\theta_P:S_n\rightarrow V$ is an isomorphism. Define $F_j:=\theta_P(S_j), X:=\theta_P\circ t\circ\theta_P^{-1}.$ Then
$$X(F_j) = \theta_P(tS_j)\subset\theta_P(S_{b_P(j)})=F_{b_P(j)}.$$
This shows $(F_\bullet,X)\in \{(F_\bullet,X)\in\text{Fl}(V)\times\End(\mathbf{C}^n): X(F_j)\subset F_{b_P(j)}\}$, which identifies with $G\times^B\mathfrak{u}_P$ by \ref{flag description of cotangent bundle}. 

Conversely, let $(F_\bullet,X)\in G\times^B\mathfrak u_P$. Since
$b_P(j)\leq j-1$, we have $X(F_j)\subset F_{j-1}$ and hence $X^n=0$.
Define
$$l_X(v):=\sum_{m=1}^nt^{-m}X^{m-1}v\in Q.$$
Then $\operatorname{pr}_{-1}\circ l_X=\operatorname{id}_V$ and, since
$X^n=0$,
\[
t l_X(v)=\sum_{m=1}^{n-1}t^{-m}X^mv=l_X(Xv).
\]
Define $S_j:=l_X(F_j)$. Since $\operatorname{pr}_{-1}\circ l_X=\operatorname{id}_V$, the map $\theta_P=\operatorname{pr}_{-1}\vert_{S_n}:S_n\rightarrow V$ is an isomorphism with inverse $l_X$. Furthermore,
$$tS_j=l_X(XF_j)\subset l_X(F_{b_P(j)})=S_{b_P(j)}\Rightarrow S_\bullet\in Y_P.$$
Since $\theta_P$ is an isomorphism, $S_\bullet$ in fact lands in $D(\sigma_P)$. Moreover, these morphisms are algebraic since $\theta_P^{-1}$ is regular on
$D(\sigma_P)$ and $l_X$ is polynomial in $X$. And they are inverse: if
$s=\sum_{m=1}^nt^{-m}v_m\in S_n$, then $v_m=\theta_P(t^{m-1}s)=X^{m-1}\theta_P(s)$,
and hence $s=l_X(\theta_P(s))$. The other inverse is checked similarly.

Finally, the universal rank-$n$ bundle and the map $\theta_B$ restrict
along $\iota$ to $\mathcal S_n$ and $\theta_P$, respectively. Taking
determinants gives $\iota^*\mathcal A_B\simeq\mathcal A_P,\text{ and }\iota^*\sigma_B=\sigma_P.$
\end{proof}

 Next, suppose $\lambda=\sum_{j=1}^{n-1}a_j\varpi_j\in\mathbf{X}^+$, so $a_j\geq 0$, and suppose $r\geq 0$. Define the following line bundles on $Y_P$:
\begin{align}
     \mathcal{L}_{P,j}&:=\det(\mathcal{S}_j)^\vee,\\
    \mathcal{M}_{P}(\lambda,r)&:=\bigotimes_{j=1}^{n-1}\mathcal{L}_{P,j}^{\otimes a_j}\otimes\mathcal{A}_P^{\otimes r}
\end{align}

\begin{lemma}\label{restriction of M_P to open}Consider the natural projection
$\pi:D(\sigma_P)\rightarrow G/B$. Then 
   $\mathcal M_P(\lambda,r)\vert_{D(\sigma_P)}
\simeq
\pi^*\mathcal L_{G/B}(\lambda)^*.$
\end{lemma}
\begin{proof}
Let $\mathcal{F}_j$ denote the tautological rank $j$ vector bundle on $G/B$. So $\det(\mathcal{F}_j)^\vee\simeq\mathcal{L}(\varpi_j)^*$ and $\mathcal{S}_j\vert_{D(\sigma_P)}\simeq\pi^*(\mathcal{F}_j)$. This implies $\mathcal{L}_{P,j}\vert_{D(\sigma_P)}\simeq\pi^*(\mathcal{L}(\varpi_j)^*)$. Moreover, on $D(\sigma_P)$, multiplication by $\sigma_P$ produces a trivialization 
$\mathcal{O}_{D(\sigma_P)}\simeq\mathcal{A}_P\vert_{D(\sigma_P)}$. Thus we deduce $\mathcal{M}_{P}(\lambda,r)\vert_{D(\sigma_P)}
\simeq
\pi^*\mathcal{L}_{G/B}(\lambda)^*$ for all $r\geq0$.
\end{proof}

\subsection{Bott-Samelson Resolution}
Let us recall some notation for Bott--Samelson varieties associated to affine Kac--Moody groups \cite{Kumar2002}. Let $\hat{G}$ denote the Kac-Moody group obtained by a central extension:
$$1\rightarrow\mathbf{G}_m\rightarrow\hat{G}\rightarrow\mathrm{SL}_n(\mathbf{C}((t)))\rightarrow 1.$$
It is normalized so that the central $\mathbf{G}_m$ acts with weight $1$ on the determinant line. The cocharacter lattice of the central $\mathbf G_m$ is generated by $c=\alpha_0^\vee+\cdots+\alpha_{n-1}^\vee$. Moreover, let $\hat{\varpi}_i,0\leq i\leq n-1$, denote the affine fundamental weights \footnote{We deviate from the standard $\Lambda_i$ notation since that is used for lattices}. They satisfy the relation $\hat{\varpi}_i=\hat{\varpi}_0+\varpi_i$ for $1\leq i<n$ and are characterized by the property that 
$$\langle\hat{\varpi}_i,\alpha_j^\vee\rangle=\delta_{ij},\quad \text{for all }0\leq i,j<n.$$

For $0\leq r\leq n$, define the standard lattices
\begin{equation}
    \Lambda_r^0:=E+t^{-1}\operatorname{span}_{\mathbf{C}[[t]]}(e_1,\dots,e_r).
\end{equation}
Thus $\Lambda_0^0=E$ and $\Lambda_n^0=t^{-1}E$. We let $\hat{G}$ act on lattices through its projection to $SL_n(\mathbf{C}((t)))$, so the central $\mathbf{G}_m$ acts trivially. Define the standard affine Iwahori subgroup and the minimal parahoric subgroup by
\begin{align}\label{lattice B and P}
    \widehat B
    &:=
    \left\{
        g\in \hat{G}:
        g\Lambda_r^0=\Lambda_r^0
        \text{ for every }0\leq r<n
    \right\},\\
    \widehat P_i
    &:=
    \left\{
        g\in \hat{G}:
        g\Lambda_r^0=\Lambda_r^0
        \text{ for every }0\leq r<n,\ r\neq i
    \right\},
    \qquad 0\leq i<n.
\end{align}
Then $\widehat{P}_i/\widehat{B}
    \simeq
    \mathbf{P}(\Lambda_{i+1}^0/\Lambda_{i-1}^0)
    \simeq\mathbf{P}^1$ for $i\neq 0$ and $\widehat P_0/\widehat B
    \simeq
    \mathbf P\bigl(\Lambda_1^0/t\Lambda_{n-1}^0\bigr)  \simeq\mathbf P^1.$ We call $\hat{G}/\hat{B}$ the affine flag variety; it may be identified with complete lattice flags. Moreover, it has a projection to the affine Grassmannian by remembering $\Lambda_0$ and has fibers $G/B$ \cite[Sect. 2.1.7]{Yun2017}.

Following \cite[Sect. 7.1.3]{Kumar2002}, given a finite affine word
$\mathbf{i}:=(i_1,\dots,i_N)$, define the corresponding Bott--Samelson
variety by
\begin{equation}\label{Bott-Samelson}
    Z(\mathbf{i})
    :=
    \widehat{P}_{i_1}\times^{\widehat{B}}
    \cdots
    \times^{\widehat{B}}\widehat{P}_{i_N}/\widehat{B}.
\end{equation}

For each $j$, we now construct a reduced word $c_j(P)$ whose
Bott--Samelson variety maps birationally to the projective space
$\mathbf P(U_j)$ arising from the $j$-th stage of the tower defining
$Y_P$ (Lemma \ref{compactification lemma}). We then concatenate these words to obtain $\mathbf{i}_P:=(c_1(P),\dots,c_n(P))$ and subsequently define $Z_P:=Z(\mathbf{i}_P)$.

Fix $1\leq j\leq n$, and set $$m_j:=n-j+b_P(j).$$ Since $P$ is fixed, we suppress it from the new notation below. Since $0\leq b_P(j)<j$, the vector space
\begin{equation*}
    U_{j}:=t^{-1}\Lambda_{b_P(j)}^0/\Lambda_{j-1}^0
\end{equation*}
has dimension $m_j+1$. It has a basis
\begin{equation}
\begin{aligned}
    u_0:=t^{-1}e_j,\;\;
    u_1:=t^{-1}e_{j+1},\dots,\;\;
    u_{n-j}:=t^{-1}e_n,\;\;
    u_{n-j+1}:=t^{-2}e_1,\;\;\dots,\;\;
    u_{m_j}:=t^{-2}e_{b_P(j)}.
\end{aligned}
\end{equation}

Define the word
\begin{equation}\label{c_j(P)}
    c_j(P):=(\beta_{m_j},\dots,\beta_1)
    :=
    (b_P(j)-1,\dots,1,0,n-1,\dots,j),
\end{equation}
where either descending string is omitted when empty. With respect to the ordered basis $u_0,\dots,u_{m_j}$, the roots $\beta_1,\dots,\beta_{m_j}$ are the adjacent simple roots. Hence they form a proper connected subdiagram of type $A_{m_j}$, and the corresponding semisimple subgroup $G_j$ acts on $U_j$ as $G_j\simeq\mathrm{SL}(U_j)\simeq\mathrm{SL}_{m_j+1}.$

Set $B_j:=G_j\cap\widehat B$ and $\ell_j^0:=\Lambda_j^0/\Lambda_{j-1}^0=\mathbf C u_0.$
The stabilizer $R_j:=\operatorname{Stab}_{G_j}(\ell_j^0)$ is the standard maximal parabolic whose Levi has simple roots $\beta_2,\dots,\beta_{m_j}$. Hence
\[
    G_j/R_j\xrightarrow{\sim}\mathbf P(U_j),
    \qquad
    gR_j\longmapsto g\ell_j^0.
\]

The Weyl group of $G_j,$ and of the Levi of $R_j$, is respectively
\[
    \langle s_{\beta_1},\dots,s_{\beta_{m_j}}\rangle
    \simeq S_{m_j+1},
    \qquad
    \langle s_{\beta_2},\dots,s_{\beta_{m_j}}\rangle
    \simeq S_{m_j}.
\]
Thus \(w_j:=s_{\beta_{m_j}}\cdots s_{\beta_1}\)
is the longest minimal coset representative, and $c_j(P)$ is a reduced
expression for $w_j$. Therefore
\[
    C_j^0:=B_jw_jR_j/R_j
\]
is the open Bruhat cell of $G_j/R_j$.

The standard Bott--Samelson construction
\cite[2.2.1]{BrionKumar2005} gives a proper birational morphism
\begin{equation}\label{phi_{P,j}}
    \phi_j:Z(c_j(P))\longrightarrow G_j/R_j\simeq\mathbf P(U_j),
    \qquad
    [q_1,\dots,q_{m_j}]
    \longmapsto(q_1\cdots q_{m_j})\ell_j^0,
\end{equation}
which is an isomorphism over $C_j^0$.

Finally, we combine the construction over all $1\leq j\leq n$ by setting (Recall \ref{Bott-Samelson})
\begin{equation}\label{Bott-Samelson for P}
    Z_P:=Z(\mathbf i_P),\;\text{ where }\mathbf i_P:=(c_1(P),\dots,c_n(P))\text{ is the concatenation of words}.
\end{equation}

\begin{prop}\label{def of rho}
The Bott--Samelson variety $Z_P$ defined in
\ref{Bott-Samelson for P} admits a proper birational morphism $\rho: Z_P \longrightarrow Y_P$ to the lattice compactification $Y_P$ defined in
\ref{Y_P definition}.
\end{prop}
\begin{proof}
Set $k_0=0$ and $k_j:=m_1+\cdots+m_j$ for $0\leq j\leq n$. Denote a point of $Z_P$ by $[q_1,\dots, q_{k_n}]$. Define $g_j:=q_1q_2\dots q_{k_j}$ and $\Lambda_j:=g_j\Lambda_j^0$. This lattice is well-defined because under the twisted-product equivalence relation, each $g_j$ changes only by right multiplication by an element of $\hat B$, and this stabilizes $\Lambda_j^0$. The nodes occurring in $c_j(P)$ are precisely the nodes outside $\{b_P(j),\dots,j-1\}$. Hence every factor in the $j$th block fixes $\Lambda_{j-1}^0$, and therefore
$$\Lambda_{j-1}=g_{j-1}\Lambda_{j-1}^0=g_j\Lambda_{j-1}^0\subset g_j\Lambda_{j}^0=\Lambda_j.$$
Similarly, every factor in the blocks
$c_{b_P(j)+1}(P),\dots,c_j(P)$ fixes $\Lambda_{b_P(j)}^0$, and consequently
$$g_j\Lambda_{b_P(j)}^0=g_{b_P(j)}\Lambda_{b_P(j)}^0=\Lambda_{b_P(j)}.$$
Furthermore, multiplication by $t$ commutes with $g_j$ and the standard lattice clearly satisfies $t\Lambda_j^0\subset\Lambda_{b_P(j)}^0$. Together these two facts imply $t\Lambda_j\subset \Lambda_{b_P(j)}$
and this shows the lattice $\Lambda_\bullet$ defines a point of $Y_P$. This construction is algebraic, and hence gives a morphism
$$
    \rho:Z_P\longrightarrow Y_P.
$$
Moreover, $\rho$ is proper because $Z_P$ is projective and $Y_P$ is separated.

Now, recall the intermediate varieties $Y_P^{(j)}$. We may similarly define $Z_P^{(j)}$ via
$$
    Z_P^{(j)}:=Z(\mathbf{i}_P^{(j)}),\;\text{ for } \mathbf{i}_P^{(j)}:=(c_1(P),\dots,c_j(P)).
$$
Thus we have in fact constructed maps $\rho^{(j)}:Z_P^{(j)}\rightarrow Y_P^{(j)}$ for all $0\leq j\leq n$. Let
$$
    \tilde{\pi}_j:Z_P^{(j)}\longrightarrow Z_P^{(j-1)},
    \qquad
    \pi_j:Y_P^{(j)}\longrightarrow Y_P^{(j-1)}
$$
be the truncation of the last factor maps. We have the following Cartesian diagram

$$\begin{tikzcd}
Z_P^{(j)} \arrow[rrd, "\rho^{(j)}", bend left] \arrow[rdd, "\tilde{\pi}_j"', bend right] \arrow[rd, "{\tilde{\phi}_{j}}", dashed] &                                                                                               &                              \\  
& Z_P^{(j-1)}\times_{Y_{P}^{(j-1)}}Y_P^{(j)} \arrow[r, "\text{pr}_2"] \arrow[d, "\text{pr}_1"'] & Y_P^{(j)} \arrow[d, "\pi_j"] \\
 & Z_P^{(j-1)} \arrow[r, "\rho^{(j-1)}"] & Y_P^{(j-1)}                 
\end{tikzcd}
$$
The universal property implies the existence of a morphism $\tilde\phi_{j}$ labeled in the diagram. Moreover, by viewing both the source and target of $\tilde\phi_{j}$ as schemes over $Z_P^{(j-1)}$ using the lower triangle of the diagram, we see that over a point $z\in Z_P^{(j-1)}$, $\text{pr}_1^{-1}(z)\simeq\mathbf{P}^{m_{j}}$ and $\tilde{\pi}_j^{-1}(z)\simeq Z(c_j(P)).$ Thus, $\tilde{\phi}_{j}$ restricted to $\tilde{\pi}_j^{-1}(z)$ may be identified with the map appearing in \ref{phi_{P,j}}:
$$\tilde{\phi}_{j}\vert_{\tilde{\pi}_j^{-1}(z)}\simeq\phi_{j}:Z(c_j(P))\rightarrow\mathbf{P}^{m_j}.$$
Recall $\phi_j$ is an isomorphism over the open Bruhat cell $C_j^0\subset \mathbf{P}^{m_j}$.By the $\hat B$-equivariance of $\phi_j$, the morphism
$\tilde{\phi}_j$ is locally $\operatorname{id}\times\phi_j$ over a
dense open subset of $Z_P^{(j-1)}$. Thus, $\tilde{\phi}_j$ is an isomorphism over a dense open subset $C_j$ of the fiber product.
We now proceed by induction. Suppose that $\rho^{(j-1)}$ is an
isomorphism over a dense open subset
$\Omega_{j-1}\subset Y_P^{(j-1)}$. Since
$\operatorname{pr}_2$ is the base change of $\rho^{(j-1)}$, it is an
isomorphism over the dense open subset
$\pi_j^{-1}(\Omega_{j-1})$, and hence is birational. Therefore $\rho^{(j)}=\operatorname{pr}_2\circ\widetilde\phi_j$
is birational. Induction up to $j=n$ proves the result.
\end{proof}

\subsection{Proof of main results}
Let us first recall the construction of marked line bundles on the Bott-Samelson varieties $Z_P$ and state their corresponding Borel-Weil-Bott theorem following \cite{LakshmibaiLittelmannMagyar2002}. We may assume $P\neq G$, since the case $P=G$ is immediate. Set
$k_0:=0$ and $k_j:=m_1+\cdots+m_j$ for $1\leq j\leq n$, and write
$\mathbf{i}_P=(i_1,\dots,i_{k_n})$. For each position $a$, denote the maximal parabolic $\hat{R}_{i_a}$ of $\hat{G}$ corresponding to omitting the simple root $\alpha_{i_a}$. Recall the configuration map
\begin{equation}\label{mu_a}
\mu_a:Z_P\rightarrow\hat{G}/\hat{R}_{i_a},\;[q_1,\dots, q_{k_n}]\mapsto (q_1\dots q_{a})\hat{R}_{i_a}
\end{equation}

Let $\mathbf{d}_P:=(d_1,\dots, d_{k_n})\in(\mathbf{Z}_{\geq0})^{k_n}$ denote a \textit{marking}. Let
\begin{equation}
    \mathcal{O}_i(1):=\hat{G}\times^{\hat{R}_i}\mathbf{C}_{-\hat{\varpi}_i},\; 0\leq i\leq n-1
\end{equation}
be the fundamental line bundles on $\hat{G}/\hat{R}_i$ and denote the marked line bundle
\begin{equation}
    \mathcal N_{\mathbf i_P,\mathbf{d}_P}
    :=
    \bigotimes_{a=1}^{k_n}
    \mu_a^*\bigl(\mathcal O_{i_a}(1)\bigr)^{\otimes d_a}.
\end{equation}

Next, following \cite[Section 1.1,4.1]{LakshmibaiLittelmannMagyar2002}, we recall the definition of the generalized Demazure module. For $0\leq i\leq n-1$, let $F_i$ be the negative Chevalley generator corresponding to $\alpha_i$, and set
\[
\mathcal U_i:=\bigoplus_{l\geq0}\mathbf C\frac{F_i^l}{l!}\subset U(\hat{\mathfrak g}).
\]
For each $1\leq a\leq k_n$, set $\lambda_a:=d_a\hat{\varpi}_{i_a}$. Let $V_{\lambda_a}$ be the corresponding integrable highest-weight module and let $v_{\lambda_a}$ be its highest-weight vector. Define the \textit{generalized Demazure module} by
\begin{equation}\label{affine Demazure module}
V_{\mathbf{i}_P,\mathbf{d}_P}
:=
\mathcal U_{i_1}\left(
v_{\lambda_1}\otimes
\mathcal U_{i_2}\left(
v_{\lambda_2}\otimes\cdots\otimes
\mathcal U_{i_{k_n}}v_{\lambda_{k_n}}
\right)\cdots
\right)
\subset
\bigotimes_{a=1}^{k_n}V_{\lambda_a},
\end{equation}
where each $\mathcal U_{i_a}$ acts diagonally on the tensor factors inside the corresponding parentheses. This is the generalized Demazure module associated with $\mathbf{i}_P$ and $\mathbf d_P$.

Consider the configuration map
\begin{equation}
\phi_P:Z_P\rightarrow
\mathbf{P}\left(\bigotimes_{a=1}^{k_n}V_{\lambda_{a}}\right),\; [q_1,\dots,q_{k_n}]
\longmapsto
\left[
q_1v_{\lambda_{1}}
\otimes(q_1q_2)v_{\lambda_{2}}
\otimes\dots\otimes
(q_1\dots q_{k_n})v_{\lambda_{k_n}}
\right].
\end{equation}
Thus $V_{\mathbf{i}_P,\mathbf{d}_P}$ may also be characterized by the property that the image of $\phi_P$ spans $\mathbf{P}(V_{\mathbf{i}_P,\mathbf{d}_P})$. Since the image of $\phi_P$ spans $\mathbf P(V_{\mathbf i_P,\mathbf d_P})$, pulling back linear forms gives an injective map
\begin{equation}\label{phi_P^*}
\phi_P^*:V_{\mathbf{i}_P,\mathbf{d}_P}^\vee
\rightarrow
H^0(Z_P,\mathcal{N}_{\mathbf{i}_P,\mathbf{d}_P})    
\end{equation}

\noindent It turns out $\phi_P^*$ is an isomorphism due to the Borel--Weil theorem for affine Bott--Samelson varieties:

\begin{thm}\cite{LakshmibaiLittelmannMagyar2002}\label{Bott-Borel for Bott-Samelson}
Let $\mathbf{i}_P=(i_1,\dots,i_{k_n})$ be a word and suppose
$\mathbf{d}_P\in(\mathbf{Z}_{\geq0})^{k_n}$. Let
$V_{\mathbf{i}_P,\mathbf{d}_P}$ denote the corresponding generalized affine Demazure module \ref{affine Demazure module}. Then
\[
H^0(Z_P,\mathcal{N}_{\mathbf{i}_P,\mathbf{d}_P})
\simeq
V_{\mathbf{i}_P,\mathbf{d}_P}^\vee,
\qquad
H^i(Z_P,\mathcal{N}_{\mathbf{i}_P,\mathbf{d}_P})=0
\quad\text{for }i\geq1.
\]
\end{thm}
Although \cite[Theorem 6]{LakshmibaiLittelmannMagyar2002} is written explicitly for finite-dimensional groups, the authors state that their results extend to symmetrizable Kac--Moody algebras \cite[p.~293]{LakshmibaiLittelmannMagyar2002}. Indeed, since $\mathbf{i}_P$ is a finite word, $Z_P$ is a finite-dimensional iterated $\mathbf P^1$-bundle, and the proof of higher-cohomology vanishing in \cite[Section 4.2]{LakshmibaiLittelmannMagyar2002} applies without change. It proceeds by induction on the length of the word and uses at each step only the rank-one calculation for $\widehat P_i/\widehat B\simeq\mathbf P^1$. The proof that $\phi_P^*$ is an isomorphism also carries over: the standard monomial basis in the Kac--Moody case is provided by \cite[Theorem 4]{Littelmann1998}, and subsequently the spanning and dimension argument of \cite[Sections 2--4]{LakshmibaiLittelmannMagyar2002} is unchanged.

Now, let $\lambda=\sum_{i=1}^{n-1}a_i\varpi_i\in\mathbf X^+$ and let
$r\in\mathbf Z_{\geq0}$. Define the corresponding marking $\mathbf{d}_P:=(d_1,\dots, d_{k_n})$ with $d_{k_j}=a_j$ for all $1\leq j<n,$ $d_{k_n}=r$, and all other $d_a$ are 0. Denote the corresponding line bundle
$$\mathcal{N}_P(\lambda,r):=\mathcal{N}_{\mathbf{i}_P,\mathbf{d}_P}.$$
\begin{lemma}\label{pullback along rho}
Let $\rho:Z_P\rightarrow Y_P$ be the map defined in Proposition \ref{def of rho}. There is an isomorphism
\begin{equation}
    \mathcal N_P(\lambda,r)
    \simeq
    \rho^*\mathcal M_P(\lambda,r).
\end{equation}
\end{lemma}
\begin{proof}
By construction of the word $\mathbf{i}_P$, we have
$i_{k_j}=j$ for $1\leq j<n$ and $i_{k_n}=0$. Thus the relevant
configuration maps are
\(
    \mu_{k_i}:Z_P\longrightarrow\widehat G/\widehat R_i
    \quad(1\leq i<n),
    \text{ and }
    \mu_{k_n}:Z_P\longrightarrow\widehat G/\widehat R_0.
\)
Identifying $\hat{G}/\hat{R}_i$ with the $\hat{G}$-orbit of the standard lattice $\Lambda_i^0$, define
\[
    \lambda_i:Y_P\longrightarrow\widehat G/\widehat R_i,
    \quad \Lambda_\bullet\longmapsto\Lambda_i
    \quad(1\leq i<n),
    \qquad
    \lambda_0:Y_P\longrightarrow\widehat G/\widehat R_0,
    \quad \Lambda_\bullet\longmapsto t\Lambda_n.
\]
We obtain the
commutative diagrams
\[
\begin{tikzcd}
Z_P \arrow[r,"\rho"] \arrow[rd,"\mu_{k_i}"']
    & Y_P \arrow[d,"\lambda_i"]\\
    & \widehat G/\widehat R_i
\end{tikzcd}
\quad(1\leq i<n),
\qquad
\begin{tikzcd}
Z_P \arrow[r,"\rho"] \arrow[rd,"\mu_{k_n}"']
    & Y_P \arrow[d,"\lambda_0"]\\
    & \widehat G/\widehat R_0.
\end{tikzcd}
\]
Indeed, writing $g_j=q_1\cdots q_{k_j}$, the first diagram follows from
$g_i\Lambda_i^0=\Lambda_i$, while the second follows from $g_nE=tg_n(t^{-1}E)=tg_n(\Lambda_n^0)=t\Lambda_n.$

Recall that two lattices $L,L'\subset V((t))$ are called commensurable if $L\cap L'$ has finite codimension in both \(L\) and \(L'\). Now, for two commensurable lattices $L$ and $L'$, define their relative determinant as
$$\text{Det}(L:L'):=\text{det}(L/M)\otimes\text{det}(L'/M)^\vee$$
where $M\subset L\cap L'$ is an arbitrary common sublattice such that $L/M$ and $L'/M$ are finite-dimensional \cite[1.5.6]{Zhu2017}. This is independent of choice of $M$. Let $\mathbf{\Lambda}_i$ be the universal bundle over $\hat{G}/\hat{R}_i$ whose fiber over $g\hat{R}_i$ is the lattice $g\Lambda_i^0$. Define the relative determinant line bundle $\text{Det}(\mathbf{\Lambda}_i:E)$ over $\hat{G}/\hat{R}_i$. We claim that 
\[
\operatorname{Det}(\mathbf\Lambda_i:E)\simeq\mathcal O_i(-1)
\qquad(0\leq i<n).
\]

Indeed, both line bundles are homogeneous $\widehat G$-bundles; thus it
suffices to compare their fibers at the standard point as
$\widehat R_i$-representations. The fibers are:
\[
\operatorname{Det}(\mathbf\Lambda_i:E)_{\Lambda_i^0}
\simeq
\begin{cases}
\det(\Lambda_i^0/E)
=
\operatorname{Span}_{\mathbf C}
(t^{-1}e_1\wedge\cdots\wedge t^{-1}e_i),
&\text{if }1\leq i<n,\\
\det(E/E)=\mathbf{C},
&\text{if }i=0.
\end{cases}
\]
Thus the ordinary torus $T$ acts with weight $\varpi_i$ when
$1\leq i<n$ and trivially when $i=0$. Since the central
$\mathbf G_m$ acts with weight one, the corresponding
$\widehat R_i$-character is $\widehat\varpi_i$ in either case. This proves the claim.

 Pulling back $\mathcal O_i(1)$ for $0\leq i<n$ and using
$E\subset\Lambda_i$ as the common sublattice for $i\neq 0$, and $tE$ as the common sublattice for $i=0$, we find, respectively,
\[
\begin{aligned}[t]
\mu_{k_i}^*(\mathcal O_i(1))
  &\simeq \rho^*\lambda_i^*(\mathcal O_i(1))
\\
  &\simeq \rho^*\!\left(
      \lambda_i^*\operatorname{Det}(\mathbf\Lambda_i:E)^\vee
    \right)
\\
  &\simeq \rho^*(\det(\mathcal S_i)^\vee)
\\
  &\simeq \rho^*(\mathcal L_{P,i})
\end{aligned}
\hspace{2.5em}
\begin{aligned}[t]
\mu_{k_n}^*(\mathcal O_0(1))
  &\simeq \rho^*\lambda_0^*(\mathcal O_0(1))
\\
  &\simeq \rho^*\lambda_0^*
      \operatorname{Det}(\mathbf\Lambda_0:E)^\vee
\\
  &\simeq \rho^*
      \operatorname{Det}(t\mathbf\Lambda_n:E)^\vee
\\
  &\simeq \rho^*\!\left(
      \det(t\mathbf\Lambda_n/tE)^\vee
      \otimes\det(E/tE)
    \right)
\\
  &\simeq \rho^*\!\left(
      \det(\mathcal S_n)^\vee\otimes\det(V)
    \right)
    \simeq \rho^*\mathcal A_P .
\end{aligned}
\]

Combining the above completes the proof:
\begin{align*}
\mathcal N_P(\lambda,r)
&=
\bigotimes_{i=1}^{n-1}
\mu_{k_i}^*\mathcal O_i(1)^{\otimes a_i}
\otimes
\mu_{k_n}^*\mathcal O_0(1)^{\otimes r}\simeq
\bigotimes_{i=1}^{n-1}
\bigl(\rho^*\mathcal L_{P,i}\bigr)^{\otimes a_i}
\otimes
\bigl(\rho^*\mathcal A_P\bigr)^{\otimes r}\simeq
\rho^*\mathcal M_P(\lambda,r).\qedhere
\end{align*}
\end{proof}

\begin{prop}\label{restriction is surjective on Y_P}
For every $G$-dominant weight $\lambda$ and every $r\geq0$, restriction induces
a surjection
\[
    H^0\bigl(Y_B,\mathcal M_B(\lambda,r)\bigr)
    \twoheadrightarrow
    H^0\bigl(Y_P,\mathcal M_P(\lambda,r)\bigr).
\]
\end{prop}
\begin{proof}
    First we compare the sections over $Y_P$ with sections over $Z_P$. By Lemma \ref{pullback along rho}, $ \mathcal N_P(\lambda,r)
    \simeq
    \rho_P^*\mathcal M_P(\lambda,r).$ Recall $\rho_P$ is also proper and birational, and $Y_P$ is smooth, hence normal. Hence $(\rho_P)_*(\mathcal{O}_{Z_P})\simeq\mathcal{O}_{Y_P}$. Thus by the projection formula,
    \begin{align*}
    H^0\bigl(Z_P,\mathcal N_P(\lambda,r)\bigr)
    &\simeq
    H^0\bigl(Z_P,\rho_P^*\mathcal M_P(\lambda,r)\bigr)\\
    &\simeq
    H^0\!\left(
       Y_P,
       \mathcal M_P(\lambda,r)\otimes
       (\rho_P)_*\mathcal O_{Z_P}
    \right)\\
    &\simeq
    H^0\bigl(Y_P,\mathcal M_P(\lambda,r)\bigr).
\end{align*}
\noindent The above computation also holds with $P$ replaced by $B$.

Next, $\mathbf{i}_P$ is a subword of $\mathbf{i}_B$, hence there is an induced closed embedding $i:Z_P\hookrightarrow Z_B$ given by inserting identity factors in the extra components. Each $c_j(P)$ is obtained from $c_j(B)$ by deleting an initial substring (see \ref{c_j(P)}), hence the marking $\mathbf{d}_P$ is obtained from the marking $\mathbf{d}_B$ by deleting zeros. Thus, $i^*\mathcal{N}_B(\lambda,r)\simeq\mathcal{N}_P(\lambda,r)$ and we obtain a natural restriction map 
$$\text{Res}:H^0(Z_B,\mathcal{N}_B(\lambda,r))\rightarrow H^0(Z_P,i^*\mathcal{N}_B(\lambda,r))=H^0(Z_P,\mathcal{N}_P(\lambda,r)).$$

Now, at a deleted position $a$, the marking $d_{B,a}=0$ and consequently the corresponding highest weight module $V_0=\mathbf{C}$. Thus, there is again a canonical inclusion $j:V_{\mathbf{i}_P,\mathbf{d}_P}\rightarrow V_{\mathbf{i}_B,\mathbf{d}_B}$ given by inserting identity factors at the trivial representation. 

By the Borel-Weil theorem, we may identify the above restriction map using the configuration maps $\phi_P,\phi_B$, as the canonical map of Demazure modules $j^*:V_{\mathbf{i}_B,\mathbf{d}_B}^\vee\rightarrow V_{\mathbf{i}_P,\mathbf{d}_P}^\vee$. Since these modules are finite-dimensional, the linear dual map is surjective and we are done. We summarize the argument via the following commutative diagram:

\[
\begin{tikzcd}[column sep=large,row sep=large]
{H^0(Y_B,\mathcal M_B(\lambda,r))}
    \arrow[r,"\operatorname{Res}"]
    \arrow[d,"\rho_B^*"']
&
{H^0(Y_P,\mathcal M_P(\lambda,r))}
    \arrow[d,"\rho_P^*"]
\\
{H^0(Z_B,\mathcal N_B(\lambda,r))}
    \arrow[r,two heads,"\operatorname{Res}"']
&
{H^0(Z_P,\mathcal N_P(\lambda,r))}
\\
{V_{\mathbf i_B,\mathbf d_B}^{\vee}}
    \arrow[r,two heads,"j^*"']
    \arrow[u,"\phi_B^*"]
&
{V_{\mathbf i_P,\mathbf d_P}^{\vee}}
    \arrow[u,"\phi_P^*"'].
\end{tikzcd}
\]
All
vertical arrows are isomorphisms, and the bottom horizontal arrow is
surjective. Therefore both restriction maps are surjective.
\end{proof}

\begin{proof}[Proof of Theorem \ref{vanishing of higher cohomology}]

By Lemma \ref{identify with T(G/P)},  it suffices to prove
\[
H^i\bigl(
G\times^B\mathfrak u_P,
\pi_B^*\mathcal L_{G/B}(\lambda)^*
\bigr)=0
\qquad(i\geq1),
\]
where $\pi_B:G\times^B\mathfrak u_P\rightarrow G/B$ is the natural projection.

By Lemmas \ref{principal open D(sigma_P)}, \ref{restriction of M_P to open}, and \ref{pullback along rho}, we have
\[
D(\sigma_P)\simeq G\times^B\mathfrak u_P,
\qquad
\mathcal M_P(\lambda,r)\vert_{D(\sigma_P)}
\simeq
\pi_B^*\mathcal L_{G/B}(\lambda)^*,
\qquad
\rho_P^*\mathcal M_P(\lambda,r)
\simeq
\mathcal N_P(\lambda,r).
\]
Since $Y_P$ is smooth, it has rational singularities. Hence, since
$\rho_P:Z_P\to Y_P$ is a resolution, $R\rho_{P*}\mathcal O_{Z_P}\simeq\mathcal O_{Y_P}.$ Using Lemma \ref{pullback along rho} and the projection formula, we thus obtain
\[
R\rho_{P*}\mathcal N_P(\lambda,r)
\simeq
R\rho_{P*}\rho_P^*\mathcal M_P(\lambda,r)
\simeq
\mathcal M_P(\lambda,r).
\]

Multiplication by $\sigma_P$ gives transition maps
\[
\mathcal M_P(\lambda,r)\longrightarrow
\mathcal M_P(\lambda,r+1).
\]
By the standard localization identity and the fact that cohomology
commutes with filtered direct limits, we obtain
\begin{equation}\label{comparison of cohomology}
\begin{aligned}
H^i\bigl(
G\times^B\mathfrak u_P,
\pi_B^*\mathcal L_{G/B}(\lambda)^*
\bigr)
&\simeq
\varinjlim_{r\geq0}
H^i\bigl(Y_P,\mathcal M_P(\lambda,r)\bigr)\\
&\simeq
\varinjlim_{r\geq0}
H^i\bigl(Z_P,\mathcal N_P(\lambda,r)\bigr).
\end{aligned}
\end{equation}
Since $\lambda$ is dominant and $r\geq0$, the marking defining
$\mathcal N_P(\lambda,r)$ is nonnegative. Therefore Theorem
\ref{Bott-Borel for Bott-Samelson} gives
\[
H^i\bigl(Z_P,\mathcal N_P(\lambda,r)\bigr)=0
\qquad(i\geq1).
\]
Taking direct limits proves the desired vanishing.
\end{proof}

\begin{proof}[Proof of Theorem \ref{generation of zeroeth cohomology}]
Set $A:=V_L(\lambda)$. By Lemma \ref{identify with T(G/P)} and setting $i=0$ in Equation \ref{comparison of cohomology}, 
\begin{equation}
\begin{aligned}
H^0\bigl(
T^*(G/P),
\mathcal L_{T^*(G/P)}(A^*)
\bigr)
&\simeq
H^0\bigl(
G\times^B\mathfrak u_P,
\pi_B^*\mathcal L_{G/B}(\lambda)^*
\bigr)\\
&\simeq
\varinjlim_{r\geq0}
H^0\bigl(Y_P,\mathcal M_P(\lambda,r)\bigr).
\end{aligned}
\end{equation}

Let $\kappa:G\times^B\mathfrak u_P \hookrightarrow G\times^B\mathfrak u$ be the natural closed embedding. The analogous direct-limit comparison of cohomology
holds with $P$ replaced by $B$. Since $\iota^*\sigma_B=\sigma_P$, the
restriction maps in Proposition
\ref{restriction is surjective on Y_P} commute with the transition maps.
Since filtered direct limits preserve surjections, we obtain a surjection
\begin{equation}\label{restriction u to u_P}
\kappa^*:
H^0\bigl(
G\times^B\mathfrak u,
\pi^*\mathcal L_{G/B}(\lambda)^*
\bigr)
\twoheadrightarrow
H^0\bigl(
G\times^B\mathfrak u_P,
\pi_B^*\mathcal L_{G/B}(\lambda)^*
\bigr),
\end{equation}
where $\pi:T^*(G/B)\rightarrow G/B$ is the natural projection.

Since $\lambda\in\mathbf{X}^+$ is dominant, \cite[Proposition 2.6]{Broer1993} gives a surjection
\begin{equation}\label{Broer 2.6}
\mathbf C[\mathfrak g^*]\otimes
H^0\bigl(G/B,\mathcal L_{G/B}(\lambda)^*\bigr)
\twoheadrightarrow
H^0\bigl(
G\times^B\mathfrak u,
\pi^*\mathcal L_{G/B}(\lambda)^*
\bigr).
\end{equation}
Combining Equations \ref{restriction u to u_P} and \ref{Broer 2.6}, we
obtain the desired surjection
\[
\mathbf C[\mathfrak g^*]\otimes
H^0\bigl(G/B,\mathcal L_{G/B}(\lambda)^*\bigr)
\twoheadrightarrow
H^0\bigl(
G\times^B\mathfrak u_P,
\pi_B^*\mathcal L_{G/B}(\lambda)^*
\bigr).\qedhere
\]
\end{proof}

\section{Generators for the coordinate ring of \texorpdfstring{$T^*(\mathrm{SL}_n/[P,P])$}{T*(SLn/[P,P])}}

\begin{lem}\label{Specialization}
Let
\[
\mathcal{M}_P
:=
G\times^P
\left(\mathfrak{g}/[\mathfrak{p},\mathfrak{p}]\right)^*,
\]
and let $\mathcal{L}_{\mathcal{M}_P}(\lambda)$ denote the pullback of
$\mathcal{L}_{G/P}(\lambda)$ along the projection $\mathcal{M}_P\rightarrow G/P$. If
$\lambda\in\mathbf{X}_{G,P}^+$, then
\[
\mathbf{C}\otimes_{\mathbf{C}[(\mathfrak{l}^{\mathrm{ab}})^*]}
\Gamma\left(
\mathcal{M}_P,\mathcal{L}_{\mathcal{M}_P}(\lambda)^*
\right)
\xrightarrow{\ \sim\ }
\Gamma\left(
T^*(G/P),\mathcal{L}_{T^*(G/P)}(\lambda)^*
\right),
\]
where $\mathbf C$ is viewed as a
$\mathbf{C}[(\mathfrak{l}^{\mathrm{ab}})^*]$-module by evaluation at the origin.
\end{lem}

\begin{proof}

Consider the short exact sequence of $P$-modules
\[
0\longrightarrow \mathfrak{l}^{\mathrm{ab}}
\longrightarrow \mathfrak{g}/[\mathfrak{p},\mathfrak{p}]
\longrightarrow \mathfrak{g}/\mathfrak{p}
\longrightarrow 0.
\]
It induces a $P$-equivariant Koszul resolution
\[
\begin{aligned}
0\longrightarrow&
\operatorname{Sym}\bigl(\mathfrak{g}/[\mathfrak{p},\mathfrak{p}]\bigr)
\otimes\bigwedge^r\mathfrak{l}^{\mathrm{ab}}
\longrightarrow\cdots\\
\longrightarrow&
\operatorname{Sym}\bigl(\mathfrak{g}/[\mathfrak{p},\mathfrak{p}]\bigr)
\otimes\mathfrak{l}^{\mathrm{ab}}
\longrightarrow
\operatorname{Sym}\bigl(\mathfrak{g}/[\mathfrak{p},\mathfrak{p}]\bigr)
\longrightarrow
\operatorname{Sym}(\mathfrak{g}/\mathfrak{p})
\longrightarrow0,
\end{aligned}
\]
where $r=\dim\mathfrak{l}^{\mathrm{ab}}$ and the $j^{\text{th}}$ term is $\operatorname{Sym}
\bigl(\mathfrak g/[\mathfrak p,\mathfrak p]\bigr)(-j)
\otimes\bigwedge^j\mathfrak l^{\mathrm{ab}}$, where \((-j)\) denotes the internal grading shift. In particular, the degree-\(m\) component is $\operatorname{Sym}^{m-j}\bigl(\mathfrak g/[\mathfrak p,\mathfrak p]\bigr)\otimes\bigwedge^j\mathfrak l^{\mathrm{ab}}.$

Tensoring with \(\mathbf C_{-\lambda}\) and applying
\(G\times^P(-)\) degree by degree gives exact complexes of
finite-rank vector bundles on \(G/P\). Since
$\mathfrak{l}^{\mathrm{ab}}$ is a trivial $P$-module, every term of
this complex is a finite direct sum, up to graded shifts, of
\begin{equation}\label{koszul terms}
G\times^P\left(
\operatorname{Sym}\bigl(\mathfrak{g}/[\mathfrak{p},\mathfrak{p}]\bigr)
\otimes\mathbf C_{-\lambda}
\right).
\end{equation}
By Broer, the higher cohomology groups of $G\times^P\left(
\operatorname{Sym}^m(\mathfrak{g}/\mathfrak{p})
\otimes\mathbf C_{-\lambda}
\right)$ vanish for every $m\geq0$. The short exact sequence of $P$-modules above gives
$\operatorname{Sym}^m(\mathfrak{g}/[\mathfrak{p},\mathfrak{p}])$
a finite filtration whose associated graded induces the isomorphism of $\mathfrak{l}$-modules
\[
\text{gr}\big(\operatorname{Sym}^m(\mathfrak{g}/[\mathfrak{p},\mathfrak{p}])\big)\simeq\bigoplus_{j=0}^m
\operatorname{Sym}^j(\mathfrak l^{\mathrm{ab}})
\otimes
\operatorname{Sym}^{m-j}(\mathfrak g/\mathfrak p).
\]

Since $\mathfrak l^{\mathrm{ab}}$ is a trivial $P$-module, each of the graded pieces of \ref{koszul terms} is acyclic. Hence the Koszul resolution is a $\Gamma$-acyclic
resolution on $G/P$. In particular, we obtain an exact sequence

\begin{align*}
    &H^0\left(G/P,G\times^P\left(
    \operatorname{Sym}^{m-1}(\mathfrak{g}/[\mathfrak{p},\mathfrak{p}])
    \otimes\mathbf C_{-\lambda}\right)\right)
    \otimes\mathfrak{l}^{\mathrm{ab}}
    \xrightarrow{\delta_1}
    H^0\left(G/P,G\times^P\left(
    \operatorname{Sym}^{m}(\mathfrak{g}/[\mathfrak{p},\mathfrak{p}])
    \otimes\mathbf C_{-\lambda}\right)\right)
    \xrightarrow{\delta_0}\nonumber\\
    &H^0\left(G/P,G\times^P\left(
    \operatorname{Sym}^{m}(\mathfrak{g}/\mathfrak{p})
    \otimes\mathbf C_{-\lambda}\right)\right)
    \longrightarrow 0.
\end{align*}

The second and third
terms are the degree-$m$ components of
\[
\Gamma\left(
\mathcal M_P,\mathcal L_{\mathcal M_P}(\lambda)^*
\right)
\quad\text{and}\quad
\Gamma\left(
T^*(G/P),\mathcal L_{T^*(G/P)}(\lambda)^*
\right),
\]
respectively. Since $\delta_1$ is multiplication, its image is the
degree-$m$ component of
\[
\mathfrak l^{\mathrm{ab}}\cdot
\Gamma\left(
\mathcal M_P,\mathcal L_{\mathcal M_P}(\lambda)^*
\right).
\]
Exactness implies that $\delta_0$ is surjective and that
$\operatorname{Ker}(\delta_0)=\operatorname{Im}(\delta_1)$. Summing
over $m$ therefore gives
\[
\Gamma\left(
T^*(G/P),\mathcal L_{T^*(G/P)}(\lambda)^*
\right)
\cong
\frac{
\Gamma\left(
\mathcal M_P,\mathcal L_{\mathcal M_P}(\lambda)^*
\right)
}{
\mathfrak l^{\mathrm{ab}}\cdot
\Gamma\left(
\mathcal M_P,\mathcal L_{\mathcal M_P}(\lambda)^*
\right)
}.
\]
This proves the desired result.
\end{proof}

\begin{thm}\label{generators of R}
The algebra $\mathbf{C}[T^*(SL_n/[P,P])]$ is generated by the image of the comoment map $\mu_{\mathfrak{sl}_n\times(\mathfrak{l}^{\mathrm{ab}})^*}: \mathbf{C}[\mathfrak{sl}_n^*\times(\mathfrak{l}^{\text{ab}})^*]\rightarrow\mathbf{C}[T^*(SL_n/[P,P])]$ and by $\Phi_w^*(\mathbf{C}[SL_n/[P^w,P^w])$, for each $w\in S_k$.
\end{thm}
\begin{proof}
Write $G=SL_n$. Consider the variety $T^*(G/[P,P]) \cong G \times^{[P,P]} \left(\mathfrak{g}/[\mathfrak{p},\mathfrak{p}]\right)^*$. We may identify $\mathfrak{g}/[\mathfrak{p},\mathfrak{p}]^* = \mathfrak{z}(\mathfrak{l})\oplus\mathfrak{u}_P$. Letting $L$ denote the Levi of $P$, we observe that the map $G \times^{[P,P]} \left(\mathfrak{g}/[\mathfrak{p},\mathfrak{p}]\right)^* \to G \times^{P} \left(\mathfrak{g}/[\mathfrak{p},\mathfrak{p}]\right)^*$ is an $L^{\textrm{ab}}$-torsor. Let $\mathcal{M}_P: = G \times^{P} \left(\mathfrak{g}/[\mathfrak{p},\mathfrak{p}]\right)^*$. 

Let $A:=\mathbf{C}[G/[P,P]]$ and $R:=\mathbf{C}[T^*(G/[P,P])]$.
The Peter-Weyl theorem for $\mathbf{C}[G]$, the Borel-Weil-Bott theorem, and the projection $T^*(G/[P,P])\rightarrow\mathcal{M}_P$ being an $L^{\mathrm{ab}}$-torsor altogether imply

\begin{align}\label{Def of A,R}
A=\mathbf{C}[G/[P,P]]&\cong\bigoplus_{\lambda\in\mathbf{X}_{G,P}^+}\Gamma(G/P,\mathcal{L}_{G/P}(\lambda)^*)\cong\bigoplus_{\lambda\in\mathbf{X}_{G,P}^+}V(\lambda)^*\\
R=\mathbf{C}[T^*(G/[P,P])] &\cong \bigoplus_{\lambda \in \mathbf{X}_{G,P}}\Gamma\left(\mathcal{M}_P, \mathcal{L}_{\mathcal{M}_P} (\lambda)^*\right).
\end{align}
where both isomorphisms are as $G\times L^{\text{ab}}$-modules. 

Given $\lambda\in\mathbf{X}_{G,P}$ and $w\in S_k$, we define $w\lambda$ by permuting the blocks. If we let $s_i:=d_i-d_{i-1}$ denote the block sizes, we may write
\[
\lambda=(\lambda_1^{s_1},\dots,\lambda_k^{s_k}),
\qquad
w\lambda=
(\lambda_{w^{-1}(1)}^{s_{w^{-1}(1)}},
 \dots,
 \lambda_{w^{-1}(k)}^{s_{w^{-1}(k)}}).
\]
where $\lambda_i^{s_i}:=(\lambda_i,\dots,\lambda_i)$ is a $s_i$-tuple. Next, define
$$R^w:= \mathbf{C}[T^*(SL_n/[P^w,P^w])],\quad A^w:=\mathbf{C}[SL_n/[P^w,P^w]].$$
\noindent By \cite{GrantcharovSlipper2026}, for each \(w\in S_k\)
there is an intertwiner
\[
\Phi_w^*:R^w\longrightarrow R
\]
that is \(w\)-twisted \(L^{\mathrm{ab}}\)-equivariant. In particular,
for every \(\lambda\in\mathbf X_{G,P}\), it induces an isomorphism
\begin{equation}\label{eq:Phi-w-weight}
\Phi_w^*:(R^w)_{w\lambda}
\xrightarrow{\sim}
R_\lambda.
\end{equation}

Now, denote the moment map \[
\mu_P:
T^*(\mathrm{SL}_n/[P,P])
\longrightarrow
\mathfrak{sl}_n^*\times(\mathfrak l^{\mathrm{ab}})^*,
\]
and let $\mu_P^*$ denote the induced comoment map on coordinate algebras. There is thus a standard map induced by multiplying the comoment map $\mu_P^*$ with the standard inclusion $A^w\hookrightarrow R^w$:
\begin{equation}\label{M_w}
M_{w,\lambda}:=\mu_{P^w}^*\otimes\operatorname{Id}_{(A^w)_\lambda},:
\mathbf{C}[\mathfrak{g}^*]
\otimes
\mathbf{C}[(\mathfrak{l}^{\mathrm{ab}})^*]
\otimes
(A^w)_\lambda
\longrightarrow
(R^w)_\lambda.
\end{equation}
We wish to show $M_{w,\lambda}$ is surjective for every dominant
$\lambda\in\mathbf{X}_{G,P^w}^+$. This suffices: for any
$\eta\in\mathbf{X}_{G,P}$, there exists $w\in S_k$ such that
$w\eta\in\mathbf{X}_{G,P^w}^+$. Moreover,
\begin{equation}\label{Phi_w_moment}
\Phi_w^*\bigl(\operatorname{im}(\mu_{P^w}^*)\bigr)
=
\operatorname{im}(\mu_P^*).
\end{equation}
Thus, after proving that Equation~\ref{M_w} is surjective for
dominant $\lambda\in\mathbf{X}_{G,P^w}^+$, we may apply
$\Phi_w^*$ to $M_{w,w\eta}$ to obtain the required generation
statement for every $\eta\in\mathbf{X}_{G,P}$.

By the graded Nakayama lemma, it suffices to prove that
$M_{w,\lambda}$ is surjective after tensoring the morphism \ref{M_w} with $(-)\otimes_{\mathbf{C}[(\mathfrak{l}^{\mathrm{ab}})^*]}\mathbf{C}.$
By Lemma~\ref{Specialization}, we may identify the resulting map with
\begin{equation}\label{surjectivity on w,lambda}
\overline{M}_{w,\lambda}:
\mathbf{C}[\mathfrak{g}^*]
\otimes
\Gamma\left(
G/P^w,\mathcal{L}_{G/P^w}(\lambda)^*
\right)
\longrightarrow
\Gamma\left(
T^*(G/P^w),
\mathcal{L}_{T^*(G/P^w)}(\lambda)^*
\right).
\end{equation}
This is surjective by
Theorem~\ref{generation of zeroeth cohomology}, thus completing the proof.
\end{proof}

\appendix
\section{Terminal singularities of the affinization of \texorpdfstring{$T^*(\mathrm{SL}_n/[P, P])$}{T*(SLn/[P, P])} (Joint with Tom Gannon)}
In this appendix, we establish the following result:

\begin{thm}\label{terminal singularities}
    The variety $\overline{T^*(SL_n/[P,P])}^{\text{aff}}$ has terminal singularities.
\end{thm}

When $P = B$, this result is all but proved in \cite{Jia2025} using the description of $\overline{T^*(SL_n/[B,B])}^{\text{aff}}$ as an $SL$-quiver variety provided in \cite{DancerKirawanSwann2013}. However, this strategy cannot be generalized for general parabolic subgroups since $\overline{T^*(SL_n/[P,P])}^{\text{aff}}$ is \textit{not} the $SL$-quiver variety associated to the analogous quiver, see \cite[Section 5.2]{GrantcharovSlipper2026}.

Instead, our proof of Theorem \ref{terminal singularities} follows a similar strategy to that used in \cite{Gannon2024}, which shows that $\overline{T^*(G/U)}^{\text{aff}}$ has terminal symplectic singularities for an arbitrary reductive group $G$. However, unlike the arguments of \cite{Gannon2024}, our argument will use the fact that $\overline{T^*(G/[P, P])}^{\text{aff}}$ has symplectic singularities for an arbitrary parabolic subgroup $P$ (originally proved in \cite{FuLiu2025}) to derive that $\overline{T^*(G/[P, P])}^{\text{aff}}$ has terminal singularities.

To prove Theorem \ref{terminal singularities}, let us fix some notation. Let $G=\mathrm{SL}_n$, and let $P$ be the standard parabolic subgroup corresponding to the partition $0=d_0<d_1<\dots<d_k=n$, so that the associated Levi subgroup $L$ has block sizes $s_i:=d_i-d_{i-1},1\leq i\leq k$. Denote the algebras $A:=\mathbf{C}[G/[P,P]]$ and $R:=\C[T^*(G/[P,P])]$ as in \ref{Def of A,R}. Given $w\in S_k$, let $P^w$ be the associate standard parabolic whose Levi subgroup is $L^w:=wLw^{-1}$. Denote the corresponding partition $0=d_0^w<d_1^w<\dots<d_k^w=n$ of $P^w$. Denote the corresponding algebras 
\begin{align*}
    A^w:=\C[SL_n/[P^w,P^w]]&\text{ and }R^w:=\C[T^*(SL_n/[P^w,P^w])]
\end{align*}
Denote the standard projection 
$$\overline{\pi}_w:\overline{T^*(G/[P^w,P^w])}^{\text{aff}}\rightarrow\overline{G/[P^w,P^w]}^{\text{aff}}.$$

Recall there is a right $L^{\text{ab}}$-action which decomposes the algebras $A^w$ and $R^w$ into weight spaces for every $w\in S_k$. For $\lambda\in\mathbf{X}_{G,P}^+$, set
\[
A_\lambda:=\Gamma(G/P,\mathcal{L}_{G/P}(\lambda)^*)\simeq V_G(\lambda)^*.
\]
We use the analogous notation for the weight spaces of $A^w$, and let $R_\lambda$ denote the $\lambda$-weight space of $R$ (See \ref{Def of A,R}). By \cite{GrantcharovSlipper2026}, there are $G$-equivariant isomorphisms for each $\lambda\in\mathbf{X}_{G,P}$:
\[
\Phi_{w,\lambda}:R_\lambda\rightarrow (R^w)_{w\lambda}.
\]
Let $\varpi_{d_i}$, $1\leq i<k$, label the fundamental weights of $\mathbf{X}_{G,P}^+$. Thus,
\[
A_{\varpi_{d_i}}\simeq V_G(\varpi_{d_i})^*
\simeq\left(\wedge^{d_i}\mathbf{C}^n\right)^*.
\]
By the Peter-Weyl theorem, the spaces $A_{\varpi_{d_i}}$, for $1\leq i<k$, generate the algebra $A$ via Cartan multiplication.

Given a global function $f$ on a variety $X$, we denote $V(f)$ the vanishing locus of $f$ and $D(f):=X\setminus V(f)$ the principal open subset of $X$. For a subset $F$ of functions on $X$, denote $D(F)=\cup_{f\in F}D(f)$ and $V(F)=\cap_{f\in F}V(f)$.

We proceed by proving a series of lemmas.
  
 \begin{lem}\label{complement ideal}
     We may identify 
     $$SL_n/[P,P]=\bigcap_{i=1}^{k-1}D(A_{\varpi_{d_i}})\subset\overline{SL_n/[P,P]}^{\text{aff}}$$
     and the complement is the vanishing of the $\delta_P:=\varpi_{d_1}+\dots+\varpi_{d_{k-1}}$-weight space of $A$:
     $$\overline{SL_n/[P,P]}^{\text{aff}}\setminus (SL_n/[P,P])=V(A_{\delta_P}).$$
 \end{lem}
\begin{proof}
By \cite[Corollary 2.4]{GrantcharovSlipper2026}, respectively \cite[Lemma 4.10]{DancerKirawanSwann2013}, we have that
\begin{align*}
    \overline{SL_n/[P,P]}^{\text{aff}}&\simeq \big(\bigoplus_{i=1}^{k-1}\Hom(\C^{d_i},\C^{d_{i+1}})\big)//\prod_{i=1}^{k-1}\mathrm{SL}_{d_i}\text{, respectively}\\
\mathrm{SL}_n/[P,P]&\simeq\big(\bigoplus_{i=1}^{k-1}\Hom_{\text{inj}}(\C^{d_i},\C^{d_{i+1}})\big)//\prod_{i=1}^{k-1}\mathrm{SL}_{d_i}
\end{align*}
where $\Hom_{\text{inj}}(V,W)$ is the open subset of injective maps. Let $$q:\bigoplus_{i=1}^{k-1}\Hom(\C^{d_i},\C^{d_{i+1}})\rightarrow\big(\bigoplus_{i=1}^{k-1}\Hom(\C^{d_i},\C^{d_{i+1}})\big)//\prod_{i=1}^{k-1}\mathrm{SL}_{d_i}$$ be the projection map. Since $A_{\varpi_{d_i}}\simeq(\wedge^{d_i}\C^n)^*$, we may identify $A_{\varpi_{d_i}}$ with the span of $d_i\times d_i$ minors of the composite $C_i:=\alpha_{k-1}\dots\alpha_{i}:\C^{d_i}\rightarrow\C^n$ where $\alpha_j\in\Hom(\C^{d_j},\C^{d_{j+1}})$. Thus,
$$q^{-1}(D(A_{\varpi_{d_i}}))=\{(\alpha_i)_{i=1}^{k-1}\in \bigoplus_{i=1}^{k-1}\Hom(\C^{d_i},\C^{d_{i+1}}):\alpha_{k-1}\circ\dots\circ\alpha_{i}:\C^{d_i}\rightarrow\C^n\text{ is injective}\}$$
$$q^{-1}(\cap_{i=1}^{k-1}D(A_{\varpi_{d_i}}))=\{(\alpha_i)_{i=1}^{k-1}\in \bigoplus_{i=1}^{k-1}\Hom(\C^{d_i},\C^{d_{i+1}}):\text{ each }\alpha_i\text{ is injective}\}$$
This implies the first claim. To prove the claim about the complement, observe that Cartan multiplication gives a surjection 
$$A_{\varpi_{d_1}}\otimes\dots \otimes A_{\varpi_{d_{k-1}}}\twoheadrightarrow A_{\delta_P}$$
Thus,
\begin{align*}
    x\notin V(A_{\delta_P})&\Leftrightarrow \text{Some $f\in A_{\delta_P}$ satisfies $f(x)\neq 0$}\\
    &\Leftrightarrow\text{ There exists $e_i\in A_{\varpi_{d_i}}$ such that $e_1(x)\dots e_{k-1}(x)\neq0$}\\
    &\Leftrightarrow x\in  D(A_{\varpi_{d_i}})\text{ for all $1\leq i\leq k-1$}.\qedhere
\end{align*}
\end{proof}

\begin{lem}\label{finite stabilizer projection}
    If a point of $\overline{T^*(SL_n/[P,P])}^{\text{aff}}$ has a finite $L^{\text{ab}}$-stabilizer, then there is some $w\in S_k$ for which it projects under $\overline{\pi}_w\circ\Phi_w$ to the open orbit $G/[P^w,P^w].$
\end{lem}
\noindent In the case $P=B$, Lemma \ref{finite stabilizer projection} follows from \cite[Cor. 5.11]{Gannon2024}. We now prove the general case.
\begin{proof}
    Suppose $p\in \overline{T^*(SL_n/[P,P])}^{\text{aff}}$ has a finite $L^{\text{ab}}$-stabilizer $T_p:=\text{Stab}_{L^{\text{ab}}}(p)$. Consider the set of weights
    $$\Sigma(p):=\{\lambda\in\mathbf{X}_{G,P}:\text{ there exists } f\in R_\lambda \text{ such that }f(p)\neq0\}.$$
    If $t\in L^{\text{ab}}$ fixes $p$ and $f(p)\neq0$, then $\lambda(t)=1$. Thus we find $T_p\subset\bigcap_{\lambda\in\Sigma(p)}\ker(\lambda)$. Conversely, if $t\in\bigcap_{\lambda\in\Sigma(p)}\ker(\lambda)$, then
$f(t\cdot p)=f(p)$ for every homogeneous $f\in R$: this is immediate
if $f(p)=0$, and if $f(p) \neq 0$ and $f$ is in homogeneous degree $\lambda$, then $\lambda \in \Sigma(p)$ and so $f(t\cdot p) = \lambda(t)f(p) = f(p)$ since $t \in \mathrm{ker}(\lambda)$. Hence $t.p=p$ and so $t \in T_p$. This argument shows
\[
T_p=\bigcap_{\lambda\in\Sigma(p)}\ker(\lambda).
\]
In particular, $T_p$ is finite if and only if $\Sigma(p)$ spans the character lattice $\mathbf{X}_{G,P}\otimes_{\mathbf{Z}}\mathbf{Q}$. Pick a basis $\lambda_1,\dots,\lambda_{k-1}\in\Sigma(p)$ of $\mathbf{X}_{G,P}\otimes_{\mathbf{Z}}\mathbf{Q}$ and choose corresponding functions $f_j\in R_{\lambda_j}$ with $f_j(p)\neq 0$. Then the cone
    \begin{equation*}\label{The cone}C := \{\sum_{j=1}^{k-1}c_j\lambda_j:c_j>0\}\end{equation*}
is open and of dimension $k-1$.

After lifting $\mu\in\mathbf{X}_{G,P}$ to a character $\widetilde{\mu}$ of the diagonal torus
of $GL_n$, write
\[
\widetilde{\mu}
=
(\underbrace{p_1^\mu,\ldots,p_1^\mu}_{s_1},
 \ldots,
 \underbrace{p_k^\mu,\ldots,p_k^\mu}_{s_k}).
\]
The numbers $p_i^\mu$ are well-defined up to adding a common integer, so
their pairwise differences are independent of the choice of lift.
The complement in $C$ of the finitely many hyperplanes $p_a^\mu=p_b^\mu$ contains a rational point. Clearing denominators, choose $c_1,\ldots,c_{k-1}\in\mathbf{Z}_{>0}$ such that $\mu=\sum_{j=1}^{k-1}c_j\lambda_j$ and the numbers $p_1^\mu,\ldots,p_k^\mu$ are pairwise distinct; this is independent of the choice of lift of $\mu$. Therefore there exists some $w \in S_k$ such that \[p_1^{w(\mu)}  > p_2^{w(\mu)}  > \dots > p_k^{w(\mu)} \] so that we may write $w(\mu) = \sum_{j = 1}^{k - 1}m_j\varpi_{d_j^w}$ for some $m_j \in \mathbf{Z}_{> 0}$. 
We deduce that the function $f:=\prod_{j=1}^{k-1}f_j^{c_j}$ is in $R_\mu$ and has the property that $f(p)=\prod f_j(p)^{c_j}\neq0$. 
    
Let $g:=\Phi_w(f)\in(R^w)_{w\mu}$ and $p^w=\Phi_w(p)$. Then $g(p^w)=f(p)\neq 0$. Since $m_i>0$ and the map in \ref{M_w} is surjective, we have 
$$g\in (R^w)_{w\mu}\subset(A^w)_{w\mu}R^w\subset(A^w)_{\varpi_{d_i^w}}R^w\text{ for all $i$. }$$ Also, since $g(p^w)\neq 0$, for every $i$ there is at least one function in $(A^w)_{\varpi_{d_i^w}}$ which is nonzero at $p^w$. Thus, by Lemma \ref{complement ideal}, $\bar{\pi}_w(p^w)\in\bigcap_{i=1}^{k-1}D((A^w)_{\varpi_{d_i^w}})=SL_n/[P^w,P^w]$ as desired.
\end{proof}

\begin{lem}\label{finite stabilizer implies smooth}
    Any point of $\overline{T^*(SL_n/[P,P])}^{\text{aff}}$ with finite $L^{\text{ab}}$-stabilizer is smooth.
\end{lem}
\begin{proof}
Suppose $p\in\overline{T^*(SL_n/[P,P])}^{\mathrm{aff}}$. By Lemma \ref{finite stabilizer projection}, there is $w\in S_k$ such that $p^w:=\Phi_w(p)$ projects via $\overline{\pi}_w$ to the smooth open subset $U:=SL_n/[P^w,P^w]\subset\overline{SL_n/[P^w,P^w]}^{\mathrm{aff}}$. We claim $\overline{\pi}_w^{-1}(U)\simeq T^*U$. Indeed, this follows from the proof of \cite[Prop 4.3]{Gannon2024}, but we can also check it directly locally. Using Lemma \ref{complement ideal}, we may find $e_i\in (A^w)_{\varpi_{d_i^w}}$ such that $e_i(p^w)\neq 0$. Let $f=e_1\dots e_{k-1}$ and consider $V_f:=D(f)\subset \overline{SL_n/[P^w,P^w]}^{\mathrm{aff}}$. Let $X:=T^*(SL_n/[P^w,P^w])$ and $X_f:=T^*(V_f)$, which is an open affine subset of $X$. Then since $X_f$ is affine, we compute
\[
X_f
\simeq \operatorname{Spec}\Gamma(X_f,\mathcal{O}_{X_f})
\simeq \operatorname{Spec}\Gamma(X,\mathcal{O}_X)_f
\simeq \operatorname{Spec}R_f^w
=\overline{\pi}_w^{-1}(V_f).
\]
Since the $V_f$ cover $SL_n/[P^w,P^w]$, we deduce $\overline{\pi}_w^{-1}(U)\simeq T^*U$, as desired.

Finally, $p^w\in\overline{\pi}_w^{-1}(U)\simeq T^*U$ and $T^*U$ is smooth implies $p^w$ is smooth. We deduce $p=\Phi_w^{-1}(p^w)$ is also smooth since $\Phi_w$ is an isomorphism.
\end{proof}

Now define $$A_{w,i}:=\Phi_w^{-1}((A^w)_{\varpi_{d_i^w}}),\;\;\; \chi_{w,i}:=w^{-1}(\varpi_{d_i^w})\in\mathbf{X}_{G,P}.$$ 
By the twisted $L^{\mathrm{ab}}$-equivariance of $\Phi_w$, every element
of $A_{w,i}$ has $L^{\mathrm{ab}}$-weight $\chi_{w,i}$.
Let us now introduce a stratification on $\overline{T^*(SL_n/[P,P])}^{\text{aff}}$. For a subset $J\subset S_k\times \{1,\dots,k-1\}$, consider the corresponding locally closed subset
$$\mathcal{S}_J:=\bigcap_{(w,i)\in J}D(A_{w,i})\cap\bigcap_{(w,i)\notin J}V(A_{w,i})\subset \overline{T^*(SL_n/[P,P])}^{\text{aff}}$$
These strata are $L^{\text{ab}}$-equivariant. In fact, we have the following:

\begin{lem}\label{constant stabilizer on strata}
   The $L^{\text{ab}}$-stabilizer is constant along strata. Namely, for any $p\in\mathcal{S}_J$, the stabilizer is 
    $$\text{Stab}_{L^{\text{ab}}}(p)=\bigcap_{(w,i)\in J}\Ker(\chi_{w,i}).$$
\end{lem}
\begin{proof}
    Suppose $p\in\mathcal{S}_J$ and $t\in L^{\text{ab}}$ fixes $p$. Then for each $(w,i)\in J$, there is some $f\in A_{w,i}$ such that $f(p)\neq 0.$ Then $f(t.p)=\chi_{w,i}(t)f(p)\neq0$. Since $t.p=p$ as well, we must have $t\in\Ker(\chi_{w,i})$, and this shows the set containment $\text{Stab}_{L^{\text{ab}}}(p)\subset\bigcap_{(w,i)\in J}\Ker(\chi_{w,i})$.

    Conversely, suppose $t\in \bigcap_{(w,i)\in J}\Ker(\chi_{w,i})$ and  $p\in\mathcal{S}_J$. We will show $f(tp)=f(p)$ for all $f\in R$ to deduce $tp=p$. By Theorem \ref{generators of R}, it suffices to check for $f\in A_{w,i}$, all $(w,i)\in S_k\times\{1,\dots,k-1\}$ and all $f$ coming from the image of the $\mathfrak{sl}_n\times\mathfrak{l}^{\text{ab}}$-moment map. If $f\in A_{w,i}$ and $(w,i)\in J$, then $f(tp)=\chi_{w,i}(t)f(p)=f(p)$ as desired. If $(w,i)\notin J$, then $p\in V(A_{w,i})$ implies $f(tp)=\chi_{w,i}(t)f(p)=0$, as desired. Finally, the moment map is $L^{\mathrm{ab}}$-equivariant for the trivial $L^{\mathrm{ab}}$-action on $\mathfrak{sl}_n^*\times(\mathfrak l^{\mathrm{ab}})^*$, since the right $L^{\mathrm{ab}}$-action commutes with the left $SL_n$-action and $L^{\mathrm{ab}}$ is abelian. Hence $f(tp)=f(p)$ for every function $f$ coming from the moment map. 
\end{proof}
\begin{lem}\label{codim 3}
    The singular locus of $\overline{T^*(SL_n/[P,P])}^{\text{aff}}$ has codimension at least 3.
\end{lem}
\begin{proof}
We first observe that every nonzero element of every $A_{w,i}$ is
prime in $R$, and that $R^\times=\C^\times$. Indeed, after applying $\Phi_w$, it is enough to prove that every nonzero element of $(A^w)_{\varpi_{d_i^w}}$ is irreducible in $R^w$. The cotangent-fiber grading on $R^w$ is nonnegative and has degree-zero part $A^w$, so any factorization of such an element in $R^w$ is already a factorization in $A^w$. The Cartan grading
\[
\deg\bigl((A^w)_{\sum_jc_j\varpi_{d_j^w}}\bigr):=\sum_jc_j
\]
is nonnegative, has degree-zero part $\C$, and places $(A^w)_{\varpi_{d_i^w}}$ in degree one. Hence one of the factors is a scalar. Since $R^w$ is a UFD \cite[Corollary 1.3(2)]{FuLiu2025} (see also \cite[Proposition 3.3]{Gannon2024}) every such element is prime. We may similarly show that $R^\times=\C^\times$.

By Lemma \ref{finite stabilizer implies smooth}, every singular point lies in a stratum with positive-dimensional stabilizer. Thus fix a subset $J\subset S_k\times \{1,\dots,k-1\}$ such that the strata $\mathcal{S}_J$ has an $L^{\text{ab}}$-stabilizer, denoted by $T_J$, with dimension at least 1. By Lemma \ref{constant stabilizer on strata},
\[
T_J=\bigcap_{(w,i)\in J}\ker(\chi_{w,i}).
\]
Let $Z$ be an irreducible component of this $\mathcal{S}_J.$

We will now successively define $\C$-algebras $S^0,S^1,S^2,S^3$ such that for each $0\leq i\leq 2$, 
$$\dim(\Spec(S^{i+1}))\leq \dim(\Spec(S^i))-1$$
and such that $Z$ is locally closed in $S^3.$ 

Let $S^0=R.$ Pick a non-trivial 1-parameter subgroup $\gamma:\C^*\rightarrow T_J.$ This induces a $\gamma$-grading on $S^0$: for $f\in R_\lambda$, declare $\deg_\gamma(f):=\langle\lambda,\gamma\rangle$. Since the characters $\chi_{1,i}=\varpi_{d_i}$ for $1\leq i<k$, form a basis of $\mathbf{X}_{G,P}$, there is some $i$ such that $d:=\langle\chi_{1,i},\gamma\rangle\neq0$. In particular, $(1,i)\notin J$ and consequently $A_{\varpi_{d_i}}$ vanishes on $Z$. Since $A_{\varpi_{d_i}}\simeq(\wedge^{d_i}\C^n)^*$ has dimension at least 2, we may pick linearly independent elements
$$a_1,a_2\in A_{1,i}= A_{\varpi_{d_i}}.$$
Now let $w_0$ denote the longest element of the $S_k$ acting on $\mathbf{X}_{G,P}$. Then $\chi_{w_0,k-i}
=w_0^{-1}\varpi_{d_{k-i}^{w_0}}
=-\varpi_{d_i}.$
Choose
\[
0\neq b\in A_{w_0,k-i}.
\]

Define $S^1:=S^0/(a_1).$ Since $a_1\in A_{\varpi_{d_i}}$, it is prime, $S^1$ is an integral domain, and $\dim(S^1)\leq \dim(S^0)-1$. Since $a_1\in A_{\varpi_{d_i}}$ vanishes on $Z$, we also have $Z$ is a locally closed subvariety of $\Spec(S^1).$ 
The images $\bar a_2,\bar b\in S^1$ are nonzero. Indeed, if either
image were zero, then $a_1$ would divide $a_2$ or $b$. Irreducibility
would make the corresponding elements associates, contradicting,
respectively, the linear independence of $a_1,a_2$ or the opposite
nonzero weights of $a_1,b$, since $R^\times=\C^\times$. Since those elements are also homogeneous with respect to the $\gamma$-grading on $S^0$, there is an induced $\gamma$-grading on $S^1$ so that $\deg_\gamma(\bar{a}_2)=d$ and $\deg_\gamma(\bar{b})=-d.$

Define $S^2:=(S^1)^{\gamma}=(S^1)_{\deg(\gamma)=0}$. This is a finitely-generated integral domain, and by considering $\gamma$-degrees, we find $S^2[t]\rightarrow S^1, t\mapsto\bar{a}_2$, is an embedding. Thus, $\dim(S^2)\leq\dim(S^1)-1$. Next, by Lemma \ref{constant stabilizer on strata}, $\gamma$ fixes any $p\in Z$ since $Z \subset \mathcal{S}_J$. This implies $Z\subset(\Spec S^1)^\gamma$. Finally, let $I_\gamma$ denote the ideal in $S^1$ generated by non-zero $\gamma$-degree elements. Then there is a surjection $S^2=(S^1)_{\deg(\gamma)=0}\twoheadrightarrow S^1/I_\gamma$. Thus, the scheme-theoretic fixed points $\Spec(S^1)^\gamma=\Spec(S^1/I_\gamma)$ is a closed subvariety of $\Spec(S^2)$ (see also \cite[Prop 4.2]{Gannon2024}), and this shows $Z$ is locally closed inside $S^2$.

Finally, define $S^3=S^2/(\bar{a}_2\bar{b})$. Note, we use $\deg_\gamma(\bar{a}_2\bar{b})=d-d=0$, so $\bar{a}_2\bar{b}\in S^2$. Furthermore, since $a_1,a_2,b$ are irreducible in $S^0$, the element $\bar{a}_2\bar{b}$ is nonzero in $S^1$. Thus, $\dim(S^3)\leq\dim(S^2)-1.$ Next, since $\bar{a}_2$ has nonzero $\gamma$-degree, it vanishes on
$(\operatorname{Spec}S^1)^\gamma$. Hence $\bar{a}_2\bar{b}$ vanishes on $Z$, and we have inclusions
$$Z\subset (\Spec(S^1))^\gamma\subset V(\bar{a}_2\bar{b})=\Spec(S^3)$$
Therefore $Z$ is a locally closed subvariety of $\Spec(S^3)$, and this completes the proof:
\begin{equation*}\dim(Z)\leq\dim(S^3)\leq\dim(S^2)-1\leq\dim(S^1)-2\leq\dim(S^0)-3\end{equation*}
Since $Z$ was arbitrary and there are only finitely many strata, the union of all strata with positive-dimensional stabilizer has codimension at least $3$. This union contains the singular locus, so the result follows.
\end{proof}

\begin{proof}[Proof of Theorem \ref{terminal singularities}]
By \cite[Corollary 1.3]{FuLiu2025}, $\overline{T^*(SL_n/[P,P])}^{\text{aff}}$ has symplectic singularities. Therefore by \cite[Corollary 1]{Nam01} to show that $\overline{T^*(SL_n/[P,P])}^{\text{aff}}$ has terminal singularities it suffices to show that the codimension of the singular locus of $\overline{T^*(SL_n/[P,P])}^{\text{aff}}$is at least 4. By Lemma \ref{codim 3}, the codimension of the singular locus is at least 3. However, by the main theorem of \cite{Nam01} (see also \cite{Kaledin2006}) the singular locus of any variety which admits symplectic singularities cannot contain an irreducible component of codimension 3. Therefore the singular locus of $\overline{T^*(SL_n/[P,P])}^{\text{aff}}$ has codimension at least four, and so the claim follows. 
\end{proof}

\bibliography{refs}{}
\bibliographystyle{alpha}

\end{document}